\documentclass{amsart}

\newtheorem{theorem}{Theorem}[section]
\newtheorem{proposition}[theorem]{Proposition}
\newtheorem{lemma}[theorem]{Lemma}
\newtheorem{corollary}[theorem]{Corollary}
\newtheorem{assumption}[theorem]{Assumption}
\newtheorem{remark}[theorem]{Remark}

\newcounter{example}
\theoremstyle{definition}
\newtheorem{example}[theorem]{Example}

\numberwithin{equation}{section}

\usepackage{paralist}
\usepackage[hidelinks]{hyperref}

\usepackage[normalem]{ulem}
\usepackage{booktabs}
\usepackage{cite}
\usepackage{enumitem}
\usepackage{float}
\usepackage{amssymb}
\usepackage{mathtools}
\usepackage{url}
\usepackage[linesnumbered,ruled,vlined]{algorithm2e}
\usepackage{csquotes}
\newcommand{\Vvert}{{\vert\kern-0.25ex\vert\kern-0.25ex\vert}}
\usepackage{tikz,pgfplots}
\usetikzlibrary{positioning,calc}
\pgfplotsset{select coords between index/.style 2 args={
		x filter/.code={
			\ifnum\coordindex<#1\fi
			\ifnum\coordindex>#2\fi
		}
}}

\usepackage{xcolor}  
\definecolor{revieworange}{RGB}{255,0,0}

\newcommand{\review}[1]{{\color{black}#1}}

\definecolor{navy}{RGB}{0,0,0}
\newcommand{\reviewtwo}[1]{{\color{black}#1}}

\newcommand{\reviewmath}[1]{\textcolor{black}{#1}}

\usepackage[colorinlistoftodos,prependcaption,textsize=tiny]{todonotes}

\newcommand\revKU[1]{{\color{black}#1}}

\newcommand{\reviewNR}[1]{{\color{black}#1}} 

\usepackage{textcomp}
\usepackage[textsize=tiny]{todonotes}

\newcommand{\cL}{\mathcal{L}}

\newcommand{\cU}{\mathcal{U}}
\newcommand{\cV}{\mathcal{V}}
\newcommand{\cW}{\mathcal{W}}
\newcommand{\cX}{\mathcal{X}}
\newcommand{\cY}{\mathcal{Y}}
\newcommand{\cZ}{\mathcal{Z}}

\newcommand{\R}{\mathbb{R}}

\newcommand{\Ext}{E}
\newcommand{\Res}{R}
\newcommand{\aExt}{E'}
\newcommand{\aRes}{R'}

\newcommand{\EBox}{\Ext_{\Omega\rightarrow\square}}
\newcommand{\RBox}{\Res_{\Omega\leftarrow\square}}

\newcommand{\tRBox}{\tilde{\Res}_{\Omega\leftarrow\square}}

\newcommand{\atRBox}{\tilde{\Res}_{\Omega\rightarrow\square}'}

\newcommand{\aEBox}{\aExt_{\Omega\leftarrow\square}}
\newcommand{\aRBox}{\aRes_{\Omega\rightarrow\square}}

\newcommand{\ECirc}{\Ext_{\mycirc\rightarrow\Omega}}
\newcommand{\RCirc}{\Res_{\mycirc\leftarrow\Omega}}

\newcommand{\aECirc}{\aExt_{\mycirc\leftarrow\Omega}}
\newcommand{\aRCirc}{\aRes_{\mycirc\rightarrow\Omega}}

\newcommand{\EstBox}{\Ext_{I\times\Omega\rightarrow I\times\square}}
\newcommand{\RstBox}{\Res_{I\times\Omega\leftarrow I\times\square}}

\newcommand{\astRBox}{\aRes_{I\times\Omega\rightarrow I\times\square}}

\newcommand{\EstCirc}{\Ext_{I\times\mycirc\rightarrow I\times\Omega}}
\newcommand{\RstCirc}{\Res_{I\times\mycirc\leftarrow I\times\Omega}}

\newcommand{\astECirc}{\aExt_{I\times\mycirc\leftarrow I\times\Omega}}

\newcommand\redsout{\bgroup\markoverwith{\textcolor{red}{\rule[0.5ex]{2pt}{0.8pt}}}\ULon}

\usepackage{wasysym}
\newcommand{\mycirc}{\text{\Circle}}

\DeclareMathOperator{\id}{id}

\DeclarePairedDelimiter{\dual}{\langle}{\rangle}
\begin{document}

\title[A posteriori certification \review{of PDE} approximations]{A posteriori certification \review{of} PDE \review{approximations with particular application to neural networks}}

\author{Lewin Ernst}
\address{Institute for Numerical Mathematics, Ulm University, Helmholtzstr. 20, 89081 Ulm, Germany}
\email{lewin.ernst@uni-ulm.de}

\author{Nikolaos Rekatsinas}
\address{Institute of Applied and Computational Mathematics, Foundation of Research and Technology, Nikolaou Plastira 100, Vassilika Vouton,
GR 70013 Heraklion, Crete, Greece}
\email{n.rekatsinas@iacm.forth.gr}

\author{Karsten Urban}
\address{Institute for Numerical Mathematics, Ulm University, Helmholtzstr. 20, 89081 Ulm, Germany}
\email{karsten.urban@uni-ulm.de}

\begin{abstract}
	We propose rigorous \revKU{and efficiently computable} lower and upper \revKU{a posteriori} error bounds for \revKU{given} approximations to PDEs \revKU{on a given domain, which might be geometrically complex. This is done} by embedding or enveloping the original domain towards geometrically simpler domains, enabling the use of fast numerical solvers. \revKU{To this end, we extend and restrict the residual and provide efficient methods to compute those Hahn-Banach extensions. Then, we efficiently compute their Riesz representations on the geometrically simpler domains and obtain the desired a posteriori bounds for which we prove that they are sharp. }

    The resulting bounds control the error in the natural norm induced by a well-posed variational formulation, require only minimal regularity assumptions, and thus remain applicable on complex geometries. The framework is detailed for elliptic as well as parabolic problems. Numerical experiments demonstrate the good quantitative behavior of the derived upper and lower error bounds. 

    \review{A central motivation for this paper comes from physics-informed and related neural-network approximations of PDEs, which are naturally mesh-free and can be evaluated pointwise on complex or parameter-dependent geometries. Nevertheless, the framework applies to any approximation for which the variational residual can be evaluated.} 
\end{abstract}

\keywords{
	A Posteriori Error Bound; Parameterized Partial Differential Equations; Extension and Restriction of Functionals; Physics Informed Neural Networks.}

\subjclass{35J20, 65M15, 68T07}

\cleardoublepage
\setcounter{page}{1}

\maketitle

\section{Introduction}
\revKU{Partial differential equations (PDEs) are well-known to model a huge variety of processes. Accordingly, there is a huge literature of numerical methods for approximately solving PDEs. On the other hand, there are several \emph{black-box} solvers which produce an approximation for the solution of a PDE without full access to the mathematical structure behind. Examples include commercial solvers or the increasing use of neural networks (NNs) for solving PDEs. The aim of this paper is to provide rigorous and efficiently computable a posteriori bounds for the error of some given approximation.} In other words, the aim of this paper is to discuss the \emph{certification} of \revKU{given} approximations to PDEs.

\revKU{Such black box solvers are particularly attractive in cases where the underlying domain $\Omega\subset\R^d$ is geometrically complex so that one might want to avoid a triangulation e.g.\ for finite element or finite volume discretizations. Other scenarios include parameterized PDEs, where e.g.\ the domain is subject to parametric changes. 

In such cases, we do not want to perform any computations on the original domain. Instead, we suggest to embed or envelop $\Omega$ by geometrically simpler domains
\begin{align}\label{eq:Domains}
    \mycirc \subset \Omega \subset \square  \subset \R^d.
\end{align}
Assuming that the PDE is well-posed, it is well-known that the dual norm of the residual is a bound for the error. In order to avoid computing this norm on $\Omega$, we extend and restrict it to $\square$ and $\mycirc$ for an upper and lower bound, respectively. This is done by efficiently computing norm-preserving Hahn-Banach extensions of the residual. Having them at hand, we efficiently compute the Riesz-representations on $\square$ and $\mycirc$, so that their norms are the desired upper and lower bound. For that procedure to work, we only need access to point evaluations of the given approximation, which is typically provided by black box solvers as well by NN approximations including physics-informed neural networks (PINNs). 
}

\revKU{The} broad and continuous interest on the usage of \revKU{NNs} for the numerical approximation of solutions to \revKU{PDEs is in fact the main motivation for this paper, see} e.g.\ \cite{Dissanayake1994,Lagaris2000a,Lagaris2000b,Raissi2018a,Raissi2019}, where this list is far from being complete. For instance, \revKU{PINNs}, trained with loss functions related to the residual of a given PDE, are extensively used to solve PDEs in particular posed on complex-shaped or varying domains, as PINNs avoid complicated discretization techniques and thus keep the implementation effort low. On the other hand, however, a theory for the assessment of the approximation quality in terms of an a posteriori error control is at least not obvious, and relatively unexplored compared to the wide acceptance of \revKU{NN} solution methods for PDEs. 

\revKU{Our point of view is that we are just given some approximation. W.r.t.\ NNs and PINNs, in} particular, we want to avoid any assumption on the architecture, training process, or training data of the neural network solution as produced by e.g.~classical PINNs \cite{Raissi2018a,Raissi2019}, Variational PINNs (VPINNs)\cite{khaVPINNs,Berrone22}, Deep Ritz methods \cite{YuB}, Deep Galerkin methods \cite{SIRIGNANO20181339}, or related approaches. We view the NN as a given black box.

A common approach to assess the accuracy of a classical PINN solution, i.e.~a solution produced by minimizing the pointwise residual of the PDE, is to bound the approximation error by the generalization error, which measures the $L_2$-norm\review{\footnote{\review{As frequently used in the literature, we denote the space of Lebesgue square-integrable functions by $L_2$ in order to distinguish from a regularity index e.g.\ in $H^s$.}}} of the residual on the entire domain, not just the training points. Such bounds typically require additional regularity of the solution and control the error in norms that are not intrinsic to the underlying PDE. In particular, controlling the $L_2$-residual is, in general, not equivalent to controlling the error in the natural norm associated with the problem, especially on complex geometries. In contrast, our approach exploits the \emph{standard error–residual relation} in dual norms (see, e.g.,~\eqref{eq:ErrorResiLinearV2}) and yields bounds directly in the natural \emph{energy} norm of the problem under minimal regularity assumptions.

In \cite{10.1093/imanum/drab093}, the generalization error has been bounded in terms of the training error, the number of training samples and the quadrature error, up to constants stemming from stability (and regularity) estimates of the underlying PDE. Although this approach provides one of the first rigorous insights for the quality of PINN approximations, it concludes that a \enquote{low generalization error is limited to a low training error}, which at least is not straightforward to realize in practice. Hence, this approach is inherently theoretical rather than of practical use for certification.~PINNs that achieve an arbitrarily small total error are constructed and analyzed for a wide class of linear parabolic problems in \cite{RMS}, and for the incompressible Navier-Stokes equations in \cite{10.1093/imanum/drac085}. The authors in \cite{10.1093/imanum/drae081} provide a framework for obtaining a posteriori estimates of \revKU{NN} solutions under continuity and coercivity of the bilinear form associated to the PDE. For a range of linear PDEs, they bound the error by the $L_2$ residual loss value (after training) in a norm under which coercivity can be established, typically by interpolating between energy and dual estimates, and which is generally weaker than the natural PDE norm. The efficacy of these a posteriori bounds then relies on a \emph{low training error}, which they achieve for several linear PDEs using specialized PINN optimization methods as in \cite{pmlr-v202-muller23b}.  

Unlike generalization error approaches, the methods in \cite{Berrone22,ernst2024certified,opschoor2024,monsuur2024} provide PDE formulations that, in turn, define loss functionals suitable for \revKU{NN} training, and which enable the computation of numerical bounds on the actual error between the \revKU{NN} approximation and the true solution. \cite{9892569} yields certification of the error without a priori knowledge of the solution using arguments from semigroup theory. \cite{Berrone22} introduces a reliable and efficient error estimator for \revKU{VPINNs}. The bound for the error consists of a residual-type term, a loss function term and a term for the data oscillation. However, the \revKU{NN} approximation should follow the corresponding discretization of the PDE. In \cite{ernst2024certified}, rigorous a posteriori error control is constructed by connecting the residual to the approximation error with computable quantities based on a well-posed formulation of the PDE. This is achieved using wavelet expansions of the dual norm of the residual as the loss function to train the \revKU{NN}. In \cite{opschoor2024}, First-Order System Least Squares (FOSLS) formulations of PDEs serve both as a loss functional and an a posteriori error estimator, allowing the least squares residual to provide a computable upper bound on the approximation error. In similar spirit, the proposed least squares formulations for PDEs in \cite{monsuur2024} provide loss functionals that yield quasi-optimal approximations.

In \cite{bachmayr2024}, the authors develop a framework for analyzing \revKU{NN} training for learning solutions to parametric PDEs. They show that if the residual loss function is derived from a stable variational formulation of the PDE, then minimizing this loss yields controlled approximation error, meaning the loss remains proportional to the error measured in a norm induced by the problem.

\subsection{Contribution}
Viewing \revKU{an approximate solution} to a PDE as produced by a black box, we propose a framework to rigorously bound the approximation error from above and from below by quantities which can efficiently be computed a posteriori -- \revKU{in particular} independently of the choice of the loss functional or the configuration of training. \revKU{As briefly described above, this} is achieved in the following way:
\begin{compactenum}[i)] 
    \item starting by a well-posed (variational) formulation of the PDE on some domain $\Omega\subset\R^d$, we use a well-known error-residual relation;
    \item choose geometrically simple domains $\mycirc$ and $\square$ such that \revKU{\eqref{eq:Domains}}; 
    \item extend and restrict the residual to $\mycirc$ and $\square$ in a suitable manner;
    \item construct computable surrogates of the dual norm of the residual on $\mycirc$ and $\square$ and relate them to the residual posed on $\Omega$.
\end{compactenum}

\revKU{We do not make use of any discretization of the underlying PDE on $\Omega$ (as e.g.\ in \cite{Berrone22}), so that the derived bounds are valid for \emph{any} approximation. By} reducing the computations for the error bounds on simpler domains $\mycirc$ and $\square$, we can use highly efficient numerical schemes.

\subsection{Organization of material}
The structure of this paper is as follows.  In Section \ref{sec:VarFormPDEs}, we recall the variational framework for weak formulations and their practical relevance to residual-based a posteriori error control. In addition, we review the evaluation of dual norms. In Section \ref{sec:CertifyPINNs}, we introduce the theoretical framework and prove the lower and upper bounds for the error. The abstract setting is detailed for elliptic and parabolic problems. Numerical experiments are presented in Section \ref{sec:NumResults} to quantitatively investigate our numerical bounds. We conclude  with a brief summary and an outlook on potential directions for future research. \revKU{In Appendix \ref{sec:PINNs}, we briefly review NNs for solving PDEs.}

\section{Error-residual relations for PDEs} \label{sec:VarFormPDEs}
In this section, we recall the main facts on error-residual relations for PDEs, which is based upon the well-known Hilbert space theory of PDEs yielding well-posed formulations.

\subsection{Well-posedness}\label{sec:wellposed}
Given a partial differential operator $B$, we require Hilbert spaces $\cW$, $\cY$ with norms $\|\cdot\|_\cX$ induced by inner products $(\cdot,\cdot)_\cX$, $\cX\in\{\cW,\cY\}$, such that  the PDE
\begin{equation} \label{eq:generalPDE}
	B u  = f
\end{equation}
is \emph{well-posed} for any appropriate right-hand side $f$, by which we mean that \eqref{eq:generalPDE} admits a unique solution, which continuously depends on the data (typically the right-hand side $f$). In order to ensure this, the domain $\cW$ and the range $\cY'$ of the operator $B$ have to be identified such that $B \in \cL_{\text{is}}(\cW, \cY')$\footnote{The space $\cL_{\text{is}}(\cW, \cY')$ is the subspace of isomorphisms from $\cL(\cW, \cY')$, where $(\cL(\cW, \cY'), \Vert \cdot \Vert_{L(\cW, \cY')})$ denotes the space of continuous linear functions from $\cW$ to $\cY'$.}, where $\cY'$ is the topological dual space of $\cY$ equipped with the operator norm 
\begin{equation}\label{eq:Opnorm}
	\Vert f \Vert_{\cY'} 
    := \sup_{v \in \cY} \frac{f(v)}{\Vert v \Vert_{\cY}}
    = \sup_{v \in \cY} \frac{\langle f, v\rangle_\cY}{\Vert v \Vert_{\cY}},
    \end{equation} 
i.e.,  $\langle \cdot,\cdot\rangle_\cY\equiv \langle \cdot,\cdot\rangle_{\cY'\times\cY}$ denotes the dual pairing. The reason why we choose the dual $\cY'$ for the range of $B$ lies in the fact that this easily allows one to relate the differential operator to a bilinear form
\begin{align*}
    b:\cW\times\cY\to\mathbb{R}
    \quad\text{via}\quad
    b(w,y) := \langle Bw,y\rangle_{\cY},
    \,\, w\in\cW, y\in\cY.
\end{align*}
 If $B\in\cL_{\text{is}}(\cW, \cY')$, the operator is bijective, which ensures existence and uniqueness for all $f\in\cY'$. Moreover, the inverse is bounded, i.e., 
\begin{align*}
    \Vert B^{-1} \Vert_{\cL(\cY', \cW)} < \infty,
\end{align*}
which will be relevant next.

Assume that we are given some approximation \revKU{$u^\theta$} of the (exact, but typically unknown) solution $u\in\cW$, e.g.\ determined by a \revKU{black box solver or a NN}. Then, we are interested in controlling the error $\Vert u - \revKU{u^\theta} \Vert_{\cW}$. If $B \in \cL_{\text{is}}(\cW, \cY')$, one can bound the error by the dual norm of the \emph{residual} $r(\revKU{u^\theta}) := f - B\revKU{u^\theta}\in\cY'$ as
\begin{equation} \label{eq:ErrorResiLinearV1}
	\Vert B \Vert^{-1}_{\cL(\cW, \cY')} \cdot \Vert \revKU{r(u^{\theta})} \Vert_{\cY'} 
    \,\le\, 
    \Vert u - \revKU{u^{\theta}} \Vert_{\cW} 
    \,\le\, 
    \Vert B^{-1} \Vert_{\cL(\cY', \cW)} \cdot 
    \Vert \revKU{r(u^{\theta})} \Vert_{\cY'}.
\end{equation}
The constants $c_B:=\Vert B \Vert_{\cL(\cW, \cY')}^{-1} >0 $ on the left and $C_B:=\Vert B^{-1} \Vert_{\cL(\cY', \cW)} <\infty$ on the right of \eqref{eq:ErrorResiLinearV1} are the (inverse of the) continuity and stability (inf-sup) constants, respectively. Then, \eqref{eq:ErrorResiLinearV1} reads
\begin{equation} \label{eq:ErrorResiLinearV2}
	c_B \, \Vert \revKU{r(u^{\theta})} \Vert_{\cY'} 
    \,\le\, 
    \Vert u - \revKU{u^{\theta}} \Vert_{\cW} 
    \,\le\, 
    C_B \, \Vert \revKU{r(u^{\theta})} \Vert_{\cY'}, \quad \forall \revKU{u^{\theta}} \in \cW.
\end{equation}
Hence, we can bound the error from above and from below by the residual, if we
\begin{compactitem}
    \item  either know or are able to compute or estimate the continuity and stability constants and
    \item are able to evaluate the dual norm of the residual, i.e., \revKU{$\Vert r(u^{\theta}) \Vert_{\cY'}$, \eqref{eq:dualnorm}.}
    However, the computation of the supremum is in general not possible.
\end{compactitem}

\subsection{PDEs on domains}
The above introduced abstract spaces $\cW$ and $\cY$ are typically function spaces which are defined upon a physical domain $\Omega \subset \mathbb{R}^d$, on which the PDE is posed. Hence, we sometimes use the notations $\cW(\Omega)$ and $\cY(\Omega)$, also for domains different from the original $\Omega$. Boundary and/or initial conditions are incorporated into the definition of the operator $B$. It will be important later to keep track on the domain $\Omega$.~\revKU{Hence, we write $B \in \cL_{\text{is}}(\cW(\Omega), \cY'(\Omega))$, and the residual reads
\begin{align}\label{eq:residual1}
    r_\Omega(u^\theta) := f - B u^\theta\in\cY'(\Omega).
\end{align}}
\begin{example}[Linear elliptic PDE]\label{example:linearPDE}
	Let $\Omega \subset \mathbb{R}^d$ be a Lipschitz domain and let $\revKU{\cW(\Omega)=\cY(\Omega)}=H^1_0(\Omega)$. Given $f\in H^{-1}(\Omega)$, the problem of finding $u \in H^1_0(\Omega)$ satisfying 
	\begin{equation*}
		\langle Bu,v\rangle_{H^1_0(\Omega)} := (A \nabla u, \nabla v )_{L_2(\Omega)} + (b \cdot \nabla u, v)_{L_2(\Omega)} + (c \cdot u, v)_{L_2(\Omega)} = f(v)
	\end{equation*}
	for all $ v \in H^1_0(\Omega)$ is well-posed, if $A \in \left(L_\infty(\Omega)\right)^{d \times d}$, $b \in \left(L_\infty(\Omega)\right)^{d}$, $c \in L_\infty(\Omega)$ and
	\begin{equation*}
		\nabla \cdot b \in L_2(\Omega), \quad c(x) 
        - \tfrac{1}{2} \nabla \cdot b(x) \ge c_0 \quad \forall x \in \Omega \; \text{ a.e. }
	\end{equation*}
	as well as 
	\begin{equation*}
		\xi^T A(x) \xi \ge a_0 \vert \xi \vert^2, \quad \forall \xi \in \mathbb{R}^{d} \; \forall x \in \Omega \; \text{ a.e.,}
	\end{equation*}	
    i.e., $A$ is \review{symmetric and positive definite}. 
    In this case \eqref{eq:ErrorResiLinearV2} holds with $c = (\Vert A \Vert_\infty + \reviewmath{C_{\text{P}}} \Vert b \Vert_\infty + \reviewmath{C_{\text{P}}}^2 \Vert c \Vert_\infty)^{-1}$ and $C = a_0^{-1}$, whereby $\reviewmath{C_{\text{P}}}$ is the Poincar\'{e} constant for $\Omega$.\hfill$\diamond$
\end{example}

\begin{example}[Space-time variational form of parabolic PDEs]\label{Ex:SpaceTime}
    Let the PDE operator $A\in\cL_{\text{is}}(W(\Omega),Y'(\Omega))$ be well-posed on some domain $\Omega\subset\mathbb{R}^d$ for appropriate function spaces $W(\Omega)$ and $Y(\Omega)$. If $A$ is elliptic, one has $W(\Omega)=Y(\Omega)=H^1_0(\Omega)$. Then, given $f\in\cY\revKU{(\Omega)}'$, with 
    \begin{align*}
    \cY\revKU{(\Omega)}'
    \kern-1pt:=L_2(I;Y'(\Omega))\kern-1pt:=\{ g:I\to Y'(\Omega): \| g\|_{L_2(I;Y'(\Omega))}^2:=\int_I \| g(t)\|_{Y'(\Omega)}^2\, dt<\infty\}, 
    \end{align*}
    we seek a solution $u\in\cW\revKU{(\Omega)}$ (with $\cW\revKU{(\Omega)}$ to be defined) of the evolution problem $$Bu:=\dot{u}+Au=f\quad \textit{in}\,\,\cY\revKU{(\Omega)}',\,\,\textit{and}\,\,u(0)=0.$$ It is well-known that the Bochner space $$\cW\revKU{(\Omega)}:=\{ w\in L_2(I;W(\Omega)): \dot{w}\in L_2(I;W'(\Omega)), w(0)=0\}$$ equipped with the graph norm $$\Vert u \Vert_{\cW\revKU{(\Omega)}}^2 := \Vert \partial_t u \Vert_{L_2(I,W'(\Omega))}^2 + \Vert u \Vert_{L_2(I,W(\Omega))}^2$$
    yields $B\in\cL_{\text{is}}(\cW\revKU{(\Omega)},\cY\revKU{(\Omega)}')$, \cite{lions2000mathematical,schwab2009space}. This is known as the space-time variational formulation of parabolic PDEs, see also \cite{CheginiStevenson,SpaceTimeUrbanPatera}. 
    \hfill$\diamond$
\end{example}

\subsection{\revKU{Residual dual} norm estimates}
We will now address the problem of finding a computable surrogate for the dual norm of the residual \revKU{\eqref{eq:residual1}} for well-posed PDE operator equations \revKU{as an a posteriori error estimator}. \revKU{Thus, we look for a feasible computable surrogate $\eta^\theta$ such that either}
\begin{subequations}\label{eq:dualnorm}
\begin{align}
	\revKU{\eta^\theta}  
        &=\Vert \revKU{r_\Omega(u^\theta)} \Vert_{\cY\revKU{(\Omega)}'}, \quad \text{or} 
        && \text{(identity)}
        \label{eq:dualnormEqual}\\
	\exists\, C,c >0:\qquad c \, \revKU{\eta^\theta}
        &\le \Vert \revKU{r_\Omega(u^\theta)} \Vert_{\cY\revKU{(\Omega)}'} 
        \le C \, \revKU{\eta^\theta}, \quad \text{or} 
        && \text{(estimate)}
        \label{eq:dualnormEst} \\
	\exists\, C > 0:\quad  \Vert \revKU{r_\Omega(u^\theta)} \Vert_{\cY\revKU{(\Omega)}'} 
        &\le C \, \revKU{\eta^\theta} 
        && \text{(bound)}
        \label{eq:dualnormUpEst}
\end{align}
\end{subequations}
holds. Ideally, we find $\revKU{\eta^\theta}$ such that identity \eqref{eq:dualnormEqual} is fulfilled, but this is typically not possible. Hence, one resorts to an estimate \eqref{eq:dualnormEst}, which ensures that $\revKU{\eta^\theta}$ encloses the \revKU{dual norm of the residual, which then} is small if and only if $\revKU{\eta^\theta}$ is. If this is still not achievable, a bound \eqref{eq:dualnormUpEst}, if not too loose, guarantees at least that the \revKU{dual norm of the residual} is below a \revKU{computable} threshold, if the bound is so.  
By applying \eqref{eq:dualnormEqual} or \eqref{eq:dualnormEst} to \eqref{eq:ErrorResiLinearV2}, the error is controlled by $\revKU{\eta^\theta}$ as
\begin{equation*}
	c \cdot \revKU{\eta^\theta} 
    \le \Vert u - \revKU{u^{\theta}} \Vert_{\cW\revKU{(\Omega)}} 
    \le C \cdot \revKU{\eta^\theta}, \quad c, C > 0
\end{equation*}
or by using \eqref{eq:dualnormUpEst} we have \revKU{at least}
\begin{equation*}
	\Vert u - \revKU{u^{\theta}} \Vert_{\cW\revKU{(\Omega)}} 
    \le C \cdot \revKU{\eta^\theta}, \quad C > 0.
\end{equation*}
Notice that the constants $c$ and $C$ in the last two equations could be different to those of \eqref{eq:ErrorResiLinearV2}.

In the next section we shall describe how to find an efficient surrogate $\revKU{\eta^\theta}$ for \revKU{any given approximation $u^\theta$}, such that one of the equations \eqref{eq:dualnorm} is satisfied. It is clear that the structure of the underlying space $\cY\revKU{(\Omega)}$ is crucial for the derivation. 

\subsection{Riesz representation} \label{subsec:RieszReps}
Since $\cY\revKU{(\Omega)}$ is a Hilbert space, the Riesz Representation Theorem\,(see e.g.\ \cite{aubin2011applied}) states that $\revKU{\eta^\theta}$ fulfilling \eqref{eq:dualnormEqual} is given by
\begin{equation*}
	\revKU{\eta^\theta} 
    := \Vert \reviewmath{\mathcal{R}}^{-1}_{\cY\revKU{(\Omega)}}\, \revKU{r_\Omega(u^\theta)} \Vert_{\cY\revKU{(\Omega)}},
\end{equation*}
where $\reviewmath{\mathcal{R}}_{\cY\revKU{(\Omega)}}: \cY\revKU{(\Omega)} \rightarrow \cY\revKU{(\Omega)}'$ denotes the Riesz-operator defined through the equation $(u, v)_{\cY\revKU{(\Omega)}} = \langle \reviewmath{\mathcal{R}}_{\cY\revKU{(\Omega)}} u, v \rangle_{\cY\revKU{(\Omega)}}$, $u, v \in \cY\revKU{(\Omega)}$, where $(\cdot,\cdot)_{\cY\revKU{(\Omega)}}$ denotes the inner product in $\cY\revKU{(\Omega)}$. The Riesz representative $\revKU{\hat{r}_\Omega(u^\theta)} := \reviewmath{\mathcal{R}}^{-1}_{\cY\revKU{(\Omega)}} \revKU{r_\Omega(u^\theta)} \in \cY\revKU{(\Omega)}$ is found by solving the equation
\begin{equation} \label{eq:RieszRep}
	(\revKU{\hat{r}_\Omega(u^\theta)}, v)_{\cY\revKU{(\Omega)}} 
    = \revKU{\dual{r_\Omega(u^\theta), v}_{\cY(\Omega)}}, \quad \forall v \in \cY\revKU{(\Omega)}.
\end{equation}
Since $\reviewmath{\mathcal{R}}_{\cY\revKU{(\Omega)}}$ is isometric we have $\revKU{\eta^\theta}  = \Vert \revKU{\hat{r}_\Omega(u^\theta)} \Vert_{\cY\revKU{(\Omega)}}$ and thus \eqref{eq:dualnormEqual}. Usually, evaluating $\Vert \cdot \Vert_{\cY\revKU{(\Omega)}}$ is much easier than $\Vert \cdot \Vert_{\cY\revKU{(\Omega)}'}$. Unfortunately, calculating the Riesz representative with \eqref{eq:RieszRep} is still unfeasible because the space $\cY\revKU{(\Omega)}$ is in almost all cases infinite-dimensional. Therefore, $\cY\revKU{(\Omega)}$ is replaced by a sufficiently high- but finite-dimensional subspace $\cY^\delta\revKU{(\Omega)} \subset \cY\revKU{(\Omega)}$, e.g., a finite element space. The Riesz representative is \revKU{thus} approximated by $\revKU{\hat{r}_\Omega^\delta(u^\theta)}  \in \cY^\delta\revKU{(\Omega)}$, which is the solution of
\begin{equation} \label{eq:RieszRepFinite}
	(\revKU{\hat{r}_\Omega^\delta(u^\theta)}, v^\delta)_{\cY\revKU{(\Omega)}} = 
    \revKU{\dual{r_\Omega(u^\theta), v^\delta}_{\cY(\Omega)}}, \quad \forall v^\delta \in \cY^\delta\revKU{(\Omega)}.
\end{equation}

\begin{remark}
If $\cY\revKU{(\Omega)}$ is not a Hilbert \revKU{space}, but only a Banach space, the Riesz Representation Theorem is not applicable. However, one way out is to use wavelet methods to directly bound the dual norm as proposed in\,\cite{ernst2024certified}. 
\end{remark}

With an increasing dimension of $\cY^\delta\revKU{(\Omega)}$ we can at least hope to have convergence of $\revKU{\hat{r}_\Omega^\delta(u^\theta)}$  to $\revKU{\hat{r}_\Omega(u^\theta)}$ in $\cY\revKU{(\Omega)}$ and therefore convergence of \revKU{the respective norms}. The drawback of this approach is that, in the case of finite elements, the space $\cY^\delta\revKU{(\Omega)}$ requires a discretization of the underlying domain $\Omega$. \revKU{Moreover, one would need to perform computations on $\Omega$, which we want to avoid.}

\section{Certification via extensions and restrictions} \label{sec:CertifyPINNs}
Now we are in position to describe the proposed approach towards deriving efficiently computable lower and upper bounds for the error $\|u-u^\theta\|_{\cW\revKU{(\Omega)}}$ of \revKU{a given approximation} $u^\theta$ \revKU{to} the solution $u$ of a PDE operator equation $Bu=f$. The idea is to use the residual $r\revKU{\equiv r(u^\theta)} :=f-B\,u^\theta$ on different domains, i.e., 
\begin{align*}
    \mycirc \subset \Omega \subset \square  \subset \R^d,
\end{align*}
\revKU{then denoted by $r_\mycirc$, $r_\Omega$ and $r_\square$, respectively,} where we shall use $\mycirc$ to derive a lower and $\square$ for an upper bound. \revKU{Since $u^\theta$ is a fixed and given input, we suppress its dependence in the notation.} This approach has a couple of consequences which we will address next:
 \begin{compactitem}
    \item We need to compute the Riesz representation of \revKU{$r$} by solving \eqref{eq:RieszRepFinite}. Hence, we need to compute \revKU{$\dual{r,v^\delta}_{\cY}$} for appropriate test functions $v^\delta$ on the respective domain. This typically amounts for computing inner products, i.e., integrals. If the \revKU{approximation $u^\theta$ and thus the} residual $r$ is defined pointwise, we \revKU{can} use quadrature (e.g.\ by choosing Gauss-Lobatto points).
    \item We need to restrict and extend \revKU{$r_\Omega$ from $\Omega$} to $\mycirc$ and $\square$ in an appropriate manner.
    \item We have to choose $\mycirc$ and $\square$ in such a way that 
    \begin{enumerate}
        \item[(i)] the residuals on these domains allow for sharp error bounds and 
        \item[(ii)] the computations \revKU{can be performed highly} efficient, which typically means that the geometries of $\mycirc$ and $\square$ need to be \enquote{simple}, e.g.\ a hypersphere or -cube. Then, standard discretizations and fast solvers such as geometric full multigrid or spectral methods are available.
    \end{enumerate}
\end{compactitem}

\subsection{Extension and restriction for primal and dual spaces}
One might think that we could just consider \revKU{\eqref{eq:RieszRepFinite}} by \revKU{simply exchanging} $\Omega$ \revKU{by} $\mycirc$ and $\square$, respectively. This, however, does not match our goals. Recall that $\cY(\Omega)$ is typically a function space, e.g.\ $H^1_0(\Omega)$. Just changing $\Omega$ to $\square$ would give rise to $H^1_0(\square)$. However, the restriction of a function in $H^1_0(\square)$ to $\Omega$ is in general \emph{not} in $H^1_0(\Omega)$. Moreover, depending on the geometry of $\Omega$, an extension might be a delicate issue. Finally, we cannot hope to control the norms of the residual on $\Omega$ by those on $\mycirc$ and $\square$. To overcome this difficulty, we propose the following recipe:
\begin{compactenum}[1.]
    \item Choose a space $\cZ(\square)$, which contains all extensions of $\cY(\Omega)$ by zero and which allows for fast solvers, \review{e.g.,} one may think of $H^1_0(\square)$.
    \item Identify a subspace $\cU(\square)\subseteq\cZ(\square)$, which is isomorphic to $\cY(\Omega)$, i.e., $\cU(\square)\cong\cY(\Omega)$. 
\end{compactenum}
We are going to formalize this idea in the following assumption.
\begin{assumption} \label{assume:Isomorph} 
	Let $\Omega \subset \square \subset \R^d$ be Lipschitz domains, and let $\cY(\Omega)$ and $\cZ(\square)$ be given normed function spaces. 
    \begin{compactenum}[(a)]
    \item Then, assume that there exists a subspace $\cU(\square)\subset \cZ(\square)$ (with the same norm $\Vert \cdot \Vert_{\cZ(\square)}$ as in $\cZ(\square)$), which is isomorphic to $\cY(\Omega)$. 
    \item The isomorphism is denoted by $\EBox \in \cL_{\text{is}}(\cY(\Omega), \cU(\square))$.
    \item We denote the continuous inverse operator of $\EBox$ by $\RBox := \EBox^{-1}$.
    \hfill$\diamond$
    \end{compactenum}
\end{assumption}
Although the assumption seems to be quite strong, we will see that for a large class of spaces we can in fact construct such an isomorphism. The names $\EBox$ and $\RBox$ are chosen to emphasize that the operators are often \emph{extension} and \emph{restriction}, at least in our examples described in Section \ref{sec:NumResults}. 
 
 We note a simple immediate consequence of Assumption \ref{assume:Isomorph}.
\begin{corollary}\label{cor:Isomorphismbounds}
	If Assumption\,\ref{assume:Isomorph} holds, then  
	\begin{equation*}
		c_\square \Vert v \Vert_{\cY(\Omega)} 
        \le \Vert \EBox v \Vert_{\cZ(\square)} 
        \le C_\square \Vert v \Vert_{\cY(\Omega)}, \quad \forall v \in \cY(\Omega),
	\end{equation*}
    where $c_\square := \Vert \RBox \Vert_{\cL(\cU(\square),\cY(\Omega))}^{-1}$ and $C_\square := \Vert \EBox \Vert_{\cL(\cY(\Omega),\cU(\square))}$.
    \qed
\end{corollary}

With the operator $\EBox$, we have the desired connection between the primal spaces $\cY(\Omega)$ and $\reviewmath{\cU}(\square)$. We need such a relation also for the dual spaces. For that, we consider the adjoint operators $\aEBox\in \cL_{\text{is}}\left(\cU'(\square), \cY'(\Omega)\right)$ of $\EBox$ and $\aRBox\in \cL_{\text{is}}\left(\cY'(\Omega), \cU'(\square)\right)$ of the inverse \revKU{$\RBox$} of $\EBox$. By definition of adjoint operators, we have
\begin{align*}
	\dual{\aEBox \revKU{r}, v }_{\cY(\Omega)} 
    &= \dual{\revKU{r},\EBox v}_{\cU(\square)}, 
    && \forall v \in \cY(\Omega),  \forall \revKU{r} \in \cU'(\square)
    \quad\text{and}\\
	\dual{\aRBox \revKU{r}_\Omega, w}_{\cU(\square)} 
    &= \dual{\revKU{r}_\Omega,\RBox w}_{\cY(\Omega)}, 
    && \forall w \in \cU(\square),  \forall \revKU{r}_\Omega \in \cY'(\Omega).
\end{align*}
Furthermore, 
\begin{align*}
\Vert \EBox \Vert_{\cL(\cY(\Omega), \cU(\square))} &= \Vert \aEBox \Vert_{\cL(\cU'(\square), \cY'(\Omega))}\quad\text{and}\\
\Vert \RBox \Vert_{\cL(\cU(\square), \cY(\Omega))} &= \Vert \aRBox \Vert_{\cL(\cY'(\Omega), \cU'(\square))}.
\end{align*}
Note that the dual operators $\aRBox$ and $\aEBox$ are changing the roles of extension and restriction. In fact, $\aRBox$ is an extension operator for functionals defined on $\cY(\Omega)$ to those on $\cU(\square)$. 
Apparently, this requires the dual space $\cU'(\square)$ of $\cU(\square)$, whose norm is given by
\begin{align*}
    \| r\|_{\cU'(\square)}
    = \sup_{u\in \cU(\square)} \frac{r(u)}{\|u\|_{\cZ(\square)}}
    \revKU{\le \| r\|_{\cZ'(\square)}},
\end{align*}
which is in general \revKU{strictly} smaller than $\| r\|_{\cZ'(\square)}$ if $\cU(\square)\subsetneq \cZ(\square)$. With this norm, we have the following estimate.

\begin{lemma} \label{lemma:extBoundsBanachSpacev1}
	Let Assumption\,\ref{assume:Isomorph} hold, then 
	\begin{equation}
		c_\square \Vert \aRBox \revKU{r}_\Omega \Vert_{\cU'(\square)} 
        \le  \Vert \revKU{r}_\Omega \Vert_{\cY'(\Omega)} \le C_\square \Vert \aRBox \revKU{r}_\Omega \Vert_{\cU'(\square)} \quad \forall \revKU{r}_\Omega \in \cY'(\Omega)
	\end{equation}
    with the constants $0<c_\square\le C_\square <\infty$ defined in Corollary \ref{cor:Isomorphismbounds}.
\end{lemma}
\begin{proof}
	The lower bound is the continuity of $\aRBox$, i.e.
	\begin{equation*}
		\Vert \aRBox \revKU{r}_\Omega \Vert_{\cU'(\square)} 
        \le \Vert \aRBox \Vert_{\cL(\cY'(\Omega), \cU'(\square))} \Vert \revKU{r}_\Omega \Vert_{\cY'(\Omega)} 
        \le \tfrac1{c_\square} \Vert \revKU{r}_\Omega \Vert_{\cY'(\Omega)}.
	\end{equation*}
	The upper bound follows from the continuity of $\aEBox$ and the fact that $\id_{\cY'(\Omega)} = (\id_{\cY(\Omega)})' = (\RBox \circ \EBox)' = (\aEBox \circ \aRBox)$, which concludes the proof.
\end{proof}

The estimates in Lemma \ref{lemma:extBoundsBanachSpacev1} seem to give us the desired upper and lower bounds. However, they involve the norm in $\cU'(\square)$, i.e., we would need the space $\cU(\square)$, which turns out to be inappropriate as the following example shows.

\begin{example} 
    Let $\cY(\Omega)=H^1_0(\Omega)$ and $\cZ(\square)= H^1_0(\square)\revKU{=\cY(\square)}$. We are searching for a subspace $\cU(\square)\subset\revKU{H^1_0(\square)}$, which is isomorphic to $\revKU{H^1_0(\Omega)}$. Apparently, the extension of an element of $H^1_0(\Omega)$ \revKU{by zero} is in $H^1_0(\square)$, and the restriction of an element in $H^1_0(\square)$ having zero trace on $\partial\Omega$ is in $H^1_0(\Omega)$. Hence, we could choose $\cU(\square)$ consisting of all zero extensions of $H^1_0(\Omega)$-functions. To compute those functions, would however require a discretization of $\Omega$, which we wanted to avoid.\hfill$\diamond$
\end{example}

This latter example shows that we need to be careful with the definition of the operators $\EBox$ yielding $\cZ(\square)$ and the subspace $\cU(\square)$. We want to do all computations on $\cZ(\square)$, which means that we need a norm-preserving extension from $\revKU{r}_\Omega\in\cY'(\Omega)$ via $\aRBox \revKU{r}_\Omega \in \cU'(\square)$ to $\cZ'(\square)$ (see Figure~\ref{fig:extension}). To this end, we suggest \revKU{to use} the norm-preserving Hahn-Banach Extension Theorem (see e.g.\ \cite[Sec.\ 9.1.2,\,Th.1]{kadets2018course}), ensuring the existence of extensions $\revKU{r}_\square^{\text{HB}} \in \cZ'(\square)$ of $\aRBox \revKU{r}_\Omega \in \cU'(\square)$ with
\begin{equation*}
	\Vert \aRBox \revKU{r}_\Omega \Vert_{\cU'(\square)} 
    = \Vert \revKU{r}_\square^{\text{HB}} \Vert_{\cZ'(\square)}.
\end{equation*} 
For all other extensions of $\aRBox \revKU{r}_\Omega$, the lower bound from Lemma \ref{lemma:extBoundsBanachSpacev1} will be lost. Unfortunately, the proof of the Hahn-Banach Extension Theorem is non-constructive, i.e., the Hahn-Banach extension $\revKU{r}_\square^{\text{HB}}$ is not known in general. Thus, we describe a construction of extensions of $\aRBox \revKU{r}_\Omega$ in \S\ref{Sec:Extension}. Then, in \S\ref{Sec:ComputableExtensions}, we develop computable approximations of norm-preserving extensions. 

\begin{figure}[ht]
\centering





  






\includegraphics{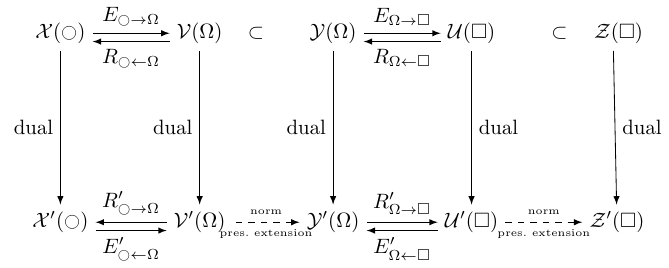}

\caption{Primal and dual extensions/restrictions. The extension $\ECirc$ embeds
the primal space $\cX(\mycirc)$ into
$\cV(\Omega)\subset\cY(\Omega)$, with inverse restriction $\RCirc$, while
$\EBox$ embeds $\cY(\Omega)$ into
$\cU(\square)\subset\cZ(\square)$, with inverse restriction $\RBox$.
The adjoint operators $\aECirc$ and $\aEBox$ extend functionals to
$\cV'(\Omega)$ and $\cU'(\square)$, respectively, while the adjoints of the restriction operators act in the reverse direction. The remaining steps
(dashed arrows) are the corresponding norm-preserving Hahn--Banach extensions to $\cY'(\Omega)$ and $\cZ'(\square)$.}
\label{fig:extension}
\end{figure}

\subsection{A lower bound}
For deriving a lower bound, we suggest to use \revKU{a} domain $\mycirc \subset \Omega$ and a corresponding function space $\cX(\mycirc)$ in such a way that Assumption \ref{assume:Isomorph} holds for the spaces $\cX(\mycirc)$ and $\cY(\Omega)$, replacing $\cY(\Omega)$ and $\cZ(\square)$. Then, there exists a subspace $\cV(\Omega) \subset \cY(\Omega)$, which is isomorphic to $\cX(\mycirc)$. We denote the corresponding extension isomorphism and its inverse (the restriction) by $\ECirc$ and $\RCirc$. Moreover, by switching the roles of $\aRCirc$ and $\aECirc$ we get the following bound.

\begin{lemma} \label{lemma:extBoundsBanachSpacev2}
	Let Assumption\,\ref{assume:Isomorph} hold \revKU{for $\mycirc \subset \Omega$ as well as $\cX(\mycirc)$, $\cY(\Omega)$, replacing $\cY(\Omega)$, $\cZ(\square)$}, such that there exists an operator $\revKU{\ECirc} \in \cL_{\text{is}}(\cX(\mycirc), \cV(\Omega))$. Then,
	\begin{equation}
		\tfrac{1}{C_\mycirc} \Vert \aECirc \revKU{r}_\Omega \Vert_{\cX'(\mycirc)} 
        \le  \Vert \revKU{r}_\Omega \Vert_{\cV'(\Omega)} 
        \le c_\mycirc \Vert \aECirc \revKU{r}_\Omega \Vert_{\cX'(\mycirc)} \quad \forall \revKU{r}_\Omega \in \cV'(\Omega)
	\end{equation}
    where $c_\mycirc := \Vert \RCirc \Vert_{\cL(\cV(\Omega),\cX(\mycirc))}$ and $C_\mycirc := \Vert \ECirc \Vert_{\cL(\cX(\mycirc),\cV(\Omega))}$. 
\end{lemma}
\begin{proof}
	The lower bound follows from the continuity of $\aECirc$ and the upper bound is derived using $\id_{\cX(\mycirc)} = \RCirc \circ \ECirc$ and the continuity of $\RCirc$.
\end{proof}

Now, we can put all pieces together and formulate the main result of this section. To this end, it is convenient to define for spaces $\cX \subset \cY$ the \emph{set of all extensions} of a given functional $\revKU{r} \in \cX'$ as
\begin{equation}\label{def:ext}
	\mathcal{E}(\revKU{r};\cY') := \big\lbrace \tilde{\revKU{r}} \in \cY': \tilde{\revKU{r}} \equiv \revKU{r} \text{ on } \cX \big\rbrace.
\end{equation}
With that at hand, we combine Lemma\,\ref{lemma:extBoundsBanachSpacev1} and Lemma\,\ref{lemma:extBoundsBanachSpacev2}.

\begin{theorem}\label{th:mainBounds} 
	Let $\mycirc \subset \Omega \subset \square~\revKU{\subset \R^d}$ and $\revKU{r}_\Omega \in \cY'(\Omega)$ be given. Furthermore, let $\EBox$, $\RBox$ and $\ECirc$, $\RCirc$ be the operators from Lemma\,\ref{lemma:extBoundsBanachSpacev1} and from Lemma\,\ref{lemma:extBoundsBanachSpacev2}, respectively. Then we have
	\begin{equation*}
		\tfrac{1}{C_{\mycirc}} \Vert \revKU{r}_\mycirc \Vert_{\cX'(\mycirc)} 
        \le \Vert \revKU{r}_\Omega \Vert_{\cY'(\Omega)} 
        \le C_\square\, \Vert \revKU{r}_\square \Vert_{\cZ'(\square)},
	\end{equation*}
	for all $\revKU{r}_\square \in \mathcal{E}(\aRBox \revKU{r}_\Omega; \cZ'(\square))$ and $\revKU{r}_\mycirc := \aECirc \left(\revKU{r}_\Omega|_{\cV(\Omega)}\right)$, where $\revKU{r}_\Omega|_{\cV(\Omega)}$ denotes the restriction of $\revKU{r}_\Omega: \cY(\Omega)\to\R$ to $\cV(\Omega)\subset\cY(\Omega)$\review{, and the constants defined as previously, i.e., $C_\mycirc := \Vert \ECirc \Vert_{\cL(\cV(\Omega),\cX(\mycirc))}$, $C_\square := \Vert \EBox \Vert_{\cL(\cY(\Omega),\cU(\square))}$}.
\end{theorem}
\begin{proof}
To derive the upper bound, we first use Lemma\,\ref{lemma:extBoundsBanachSpacev1}, which yields $\Vert \revKU{r}_\Omega \Vert_{\cY'(\Omega)} \le C_\square \Vert \aRBox \revKU{r}_\Omega \Vert_{\cU'(\square)}$. Next, the Hahn-Banach Extension Theorem ensures the existence of an extension $\revKU{r}_\square^{\text{HB}} \in \cZ'(\square)$ of $\aRBox \revKU{r}_\Omega \in \cU'(\square)$ to $\cZ'(\square)$ with the identity $\Vert \aRBox \revKU{r}_\Omega \Vert_{\cU'(\square)} = \Vert \revKU{r}_\square^{\text{HB}} \Vert_{\cZ'(\square)}$ and all other extensions have larger norms, which proves the upper bound. 
For the lower bound, \review{since $\cV(\Omega)\subset\cY(\Omega)$ is equipped with the norm inherited from $\cY(\Omega)$, every $r_\Omega\in\cY'(\Omega)$ induces by restriction a continuous functional $r_\Omega|_{\cV(\Omega)}\in\cV'(\Omega)$}, so that $\Vert \revKU{r}_\Omega |_{\cV(\Omega)}\Vert_{\cV'(\Omega)} \le \Vert \revKU{r}_\Omega \Vert_{\cY'(\Omega)}$. Finally, the lower bound in Lemma\,\ref{lemma:extBoundsBanachSpacev2} applied to $\revKU{r}_\Omega|_{\cV(\Omega)} \in \cV'(\Omega)$ completes the proof.
\end{proof}
\review{\begin{remark}
The lower bound is (only) based upon $r_\Omega|_{\cV(\Omega)}$, the restricted functional. In the isometric case
$C_\mycirc=1$, the lower estimator is exact,  precisely when $r_\Omega$ coincides with a norm-preserving extension of this restriction (cf.~Figure~\ref{fig:extension}), i.e., $\|r_\Omega|_{\cV(\Omega)}\|_{\cV'(\Omega)}
=\|r_\Omega\|_{\cY'(\Omega)}$.
\end{remark}}

\subsection{Constructing an extension}\label{Sec:Extension}
The main building blocks of the above theory are the extension isomorphism $\ECirc$ and $\EBox$, \review{together with} their duals and inverses. Of course, the extensions depend on the geometry of $\Omega$ and the specific form of the spaces $\cY(\Omega)$ and $\cZ(\square)$, which in turn are determined by the underlying PDE. We are going to detail \review{the} two relevant cases \review{used in the remainder of the paper}. 

\subsubsection{Sobolev Spaces with zero trace}\label{Subsec:Sobolev} Let $\Omega \subset \mathbb{R}^{d}$ and let $\cY(\Omega):=\reviewmath{H^m_0(\Omega)}$ be the Sobolev spaces  of order $m\in\mathbb{N}$. 

\revKU{%
\begin{remark}
All what is said below can also be formulated for $W^{m,p}_0(\Omega)$-spaces. However, for the ease of presentation,  we restrict the presentation to Hilbert spaces $H^m_0(\Omega)$, which is the setting used in this paper.
\end{remark}}

\revKU{The} \emph{zero extension} $\tilde{u}$ of $u \in  \reviewmath{H^m_0(\Omega)}$ is defined as
\begin{equation*}
	\tilde{u}(x) :=  
	\begin{cases}
		u(x), \quad & x \in \Omega, \\[2pt]
		0, \quad & x \in \Omega^{\text{c}}:=\mathbb{R}^d\setminus\Omega.
	\end{cases}
\end{equation*}

\begin{proposition}[{\cite[Lemma 3.27]{Adams}}]\label{prop:Wmp0extensions}
    Let $u \in \reviewmath{H^m_0(\Omega)}$. If $\vert \alpha \vert \le m$, then $D^\alpha \tilde{u} = \widetilde{D^\alpha u}$ in the distributional sense in $\mathbb{R}^{\reviewmath{d}}$, i.e., $\tilde{u} \in \reviewmath{H^m_0(\mathbb{R}^d)}$.
    \hfill\qed
\end{proposition}
Studying the proof of \cite[Lemma 3.27]{Adams}, we note that there are no assumptions on \revKU{the geometry of } the domain $\Omega$. In the following we denote the dual spaces of \reviewmath{$H^m_0(\Omega)$} by \reviewmath{$H^{-m}(\Omega) := (H^m_0(\Omega))'$}.

Let $u\in \cY(\Omega)=\reviewmath{H^m_0(\Omega)}$ for some $\Omega\subset\R^d$, $\tilde{u}$ be its zero extension to $\R^d$ and let $\tilde{u}|_\square$ be the restriction to $\square\supset\Omega$. Then, $\tilde{u}|_\square\in \reviewmath{H^m_0(\square)}$ and we define 
\begin{align}\label{eq:EboxSobolev}
    \EBox: \reviewmath{H^m_0(\Omega)} \rightarrow \reviewmath{H^m_0(\square)}
    \quad\text{by}\quad
    u \mapsto \EBox u := \tilde{u}|_\square,
\end{align}
which is a linear operator. Setting 
\begin{equation}\label{eq:UdefSobolev}
	\cU(\square) 
    := \EBox \left(\revKU{H_0^{m}(\Omega)} \right)  
    \subsetneq \cZ(\square) = \revKU{H_0^{m}}(\square)
    \revKU{\equiv \cY(\square)},
\end{equation}
i.e., the range of $\EBox$, Proposition\,\ref{prop:Wmp0extensions} ensures that $\EBox: \revKU{H_0^{m}}(\Omega) \rightarrow \cU(\square)$ is an isometry, i.e.,
\begin{align}\label{eq:Isometry}
    C_\square = \| \EBox\|_{\cL(\cY(\Omega),\cU(\reviewmath{\square}))}=1.
\end{align}
Hence, the operator is continuous and bounded from below. By \cite[Lemma\,2.8]{Abramovich2002}, this is equivalent to $\EBox$ being an isomorphism. The restriction $\RBox: \cU(\square) \rightarrow \revKU{H_0^{m}}(\Omega)$ is the inverse operator and we get $c_\square = \| \RBox\|_{\cL(\cY(\Omega),\cU(\reviewmath{\square}))}^{-1}=1$. Note, that the constants in Corollary \ref{cor:Isomorphismbounds} are in fact unity here, which is due to the zero trace in combination with zero extension.

We can proceed in an analogous manner for $\mycirc \subset \Omega$, $\cX(\mycirc) :=  \revKU{H_0^{m}}(\mycirc)$ and get an isomorphism $\ECirc:  \revKU{H_0^{m}}(\mycirc) \rightarrow \cV(\Omega)$ defined as $\ECirc u := \tilde{u}|_\Omega$ for 
\begin{equation}\label{eq:VdefSobolev}
	\cV(\Omega) := \ECirc \left( \revKU{H_0^{m}}(\mycirc) \right) \subsetneq \cY(\Omega).
\end{equation}
Again, we obtain isometries, i.e., the involved constants are unity. Thus, Theorem\,\ref{th:mainBounds} implies the following statement.
\begin{corollary}\label{cor:boundsSobolev}
	Let $\mycirc\subset\Omega\subset\square\subset\R^d$, $\cY(\Omega):= \revKU{H_0^{m}}(\Omega)$, 
    $\cX(\mycirc):= \revKU{\cY(\mycirc):=} \revKU{H_0^{m}}(\mycirc)$  and 
    $\cZ(\square):= \revKU{\cY(\square):=} \revKU{H_0^{m}}(\square)$ for $m\in\mathbb{N}$ and $1 \le p \le \infty$. Let $\ECirc$, $\EBox$ be defined as above and $\RCirc$, $\RBox$ being their inverses. Then, 
	\begin{equation}
		\Vert \revKU{r}_\mycirc \Vert_{\revKU{H^{-m}}(\mycirc)} \le \Vert \revKU{r}_\Omega \Vert_{\revKU{H^{-m}}(\Omega)} \le \Vert \revKU{r}_\square \Vert_{\revKU{H^{-m}}(\square)}
	\end{equation}
	for all $\revKU{r}_\square \in \mathcal{E}(\aRBox \revKU{r}_\Omega; \revKU{H^{-m}}(\square))$ and $\revKU{r}_\mycirc := \aECirc \left(\revKU{r}_\Omega|_{\cV(\Omega)}\right)$ with $\cV(\Omega)$ defined in \eqref{eq:VdefSobolev}.
    \hfill\qed
\end{corollary}

\subsubsection{Bochner spaces}\label{Subsec:Bochner}
We are now going to consider instationary parabolic problems involving space and time as already introduced in Example \ref{Ex:SpaceTime} for $Y(\Omega):=\revKU{H_0^{m}}(\Omega)$. We perform the embedding strategy only w.r.t.\ the space variable and set
\begin{align*}
    \cX(\mycirc) := L_2(I; \revKU{H_0^{m}}(\mycirc)),
    \quad
    \cY(\Omega) := L_2(I; \revKU{H_0^{m}}(\Omega)),
    \quad
    \cZ(\square) := L_2(I; \revKU{H_0^{m}}(\square)).
\end{align*}
Hence, we can reuse the findings of the previous \S\ref{Subsec:Sobolev}. Using $\EBox$ defined in \eqref{eq:EboxSobolev} and $U(\square):=\EBox \left(\revKU{H_0^{m}}(\Omega) \right)$ as in \eqref{eq:UdefSobolev}, we define a space-time extension operator 
\begin{equation*}
	\EstBox : 
    \cY(\Omega) = L_2\left(I; Y(\Omega)\right) \rightarrow L_2\left(I; U(\square)\right)
    =:\cU(\square),
\end{equation*} 
 and 
\begin{equation}\label{eq:timeExt}
	(\EstBox v)(t) := \EBox (v(t)),
    \qquad t\in I\, \text{a.e.},
    \quad
    v\in \cY(\Omega).
\end{equation}
With this definition, the following result is immediate. 

\begin{proposition}\label{prop:LisBochner}
    Let $\EBox \in \cL_{\text{is}}\left(Y(\Omega), U(\square)\right)$. Then, for $\EstBox$  defined as in \eqref{eq:timeExt}, we have $\EstBox \in \cL_{\text{is}}\left(L_2(I;Y(\Omega)), L_2(I;U(\square))\right)$ .
\end{proposition}
\begin{proof}
    By Corollary\,\ref{cor:Isomorphismbounds}, we have that the function $t \mapsto \Vert \EBox (v(t)) \Vert_{Z(\square)}$, $Z(\square):=\revKU{H_0^{m}}(\square)$, is upper-bounded by the function $t \mapsto C_\square \Vert (v(t)) \Vert_{\cY(\Omega)}$ and lower-bounded by the function $t \mapsto \frac1{c_{\square}} \Vert (v(t)) \Vert_{\cY(\Omega)}$. Thus, $\EstBox v \in L_2(I;U(\square))$ and the operator $\EstBox$ is continuous as well as bounded from below. Using again \cite[Lemma 2.8]{Abramovich2002}, it remains to show that the range of $\EstBox$ is $L_2(I;U(\square))$. Given $v \in L_2(I;U(\square))$, define $w \in L_2(I;Y(\Omega))$ by $w(t):= \RBox(v(t))$ for almost all $t \in I$. With that we have $\EstBox w = v$, which completes the proof. 
\end{proof}

The inverse operator $\EstBox^{-1}$ is denoted by $\RstBox$. Finally, we define $\EstCirc$ and $\RstCirc$ using the operators $\ECirc$, $\RCirc$ from Corollary\,\ref{cor:boundsSobolev}, i.e., $(\EstCirc v)(t):= \ECirc(v(t))$ and $\RstCirc:=\EstCirc^{-1}$. Then, we can put everything together and obtain.

\begin{corollary}\label{cor:boundsBochner}
    With $\cX(\mycirc)$, $\cY(\Omega)$, $\cZ(\square)$ and the above operators, we have for all $\revKU{r}_\square \in \mathcal{E}\left(\astRBox \revKU{r}_\Omega; L_2(I,\revKU{H^{-m}}(\square))\right)$
	\begin{equation}
		\Vert \revKU{r}_\mycirc \Vert_{L_2(I,\revKU{H^{-m}}(\mycirc))} 
        \le \Vert \revKU{r}_\Omega \Vert_{L_2(I,\revKU{H^{-m}}(\Omega))} 
        \le \Vert \revKU{r}_\square \Vert_{L_2(I,\revKU{H^{-m}}(\square))},
	\end{equation}
	setting $\revKU{r}_\mycirc := \astECirc \left(\revKU{r}_\Omega|_{L_2(I,\cV'(\Omega))}\right)$.
    \hfill\qed
\end{corollary}

\subsection{Approximation of a norm-preserving extension.}\label{Sec:ComputableExtensions} The upper bound in Theorem~\ref{th:mainBounds} requires \revKU{a norm-preserving} extension of $\aRBox \revKU{r}_\Omega\revKU{: \cU(\square)\to\R}$ in $\cU'(\square)$ to an element of $\cZ'(\square)$\revKU{, i.e., a linear functional on $\cZ(\square)\supset\cU(\square)$} (see also Figure~\ref{fig:extension}). The Hahn-Banach Extension \revKU{T}heorem ensures the existence of \revKU{such an extension} $\revKU{r}_\square^{\text{HB}} \in \cZ'(\square)$ with $\Vert \aRBox \revKU{r}_\Omega \Vert_{\cU'(\square)} = \Vert \revKU{r}_\square^{\text{HB}} \Vert_{\cZ'(\square)}$. However, \revKU{as already mentioned, } the theorem provides only an existence statement but it does not yield an explicit construction of such an extension.

In a Hilbert space setting, the norm-preserving extension $\revKU{r}_\square^{\text{HB}}$ can be constructed explicitly through the Riesz \revKU{Representation Theorem. However,} computing it would implicitly require resolving the geometry of $\Omega$, which contradicts our goal of performing all computations solely on the simple \revKU{domains $\mycirc$ and $\square$}. In this section, we develop computable realizations of such extensions \revKU{acting only on $\mycirc$ and $\square$, respectively}, together with explicit bounds that quantify the approximation quality of the proposed extension. 

\subsubsection{Optimal norm--preserving extensions in Hilbert spaces} We first recall that in Hilbert spaces the Hahn--Banach extension can be characterized explicitly through orthogonal projection. \revKU{Moreover, this extension is optimal in the sense that its norm is minimal w.r.t.\ all possible extensions.}

\begin{lemma}
\label{lem:HB-Hilbert-extension}
Let $\cZ(\square)$ be a real Hilbert space with inner product $(\cdot,\cdot)_\cZ$ and induced norm
$\|\cdot\|_{\cZ(\square)}$ and let Assumption \ref{assume:Isomorph} hold. Denote with $P_\cU: \cZ(\square)\to \cU(\square)$ the $\cZ(\square)$-orthogonal projection onto
$\cU(\square)$. For any $\reviewtwo{r}\in \cU'(\square)$, define the \revKU{Hahn-Banach} extension $\revKU{r}^{\text{HB}}_\square\in \cZ'(\square)$ by
\begin{equation}\label{eq:optextension}
    \revKU{r}^{\text{HB}}_\square(z) := \reviewtwo{r}\left(P_\cU z\right), \qquad z\in \cZ(\square),
\end{equation}
and $\revKU{r}^{\text{HB}}_\square$ is \revKU{the optimal} continuous extension of $\reviewtwo{r}$ from $\cU'(\square)$
to $\cZ'(\square)$ \revKU{in the sense that}
\begin{align*}
\|\revKU{r}^{\text{HB}}_\square\|_{\cZ'(\square)} = \|\reviewtwo{r}\|_{\cU'(\square)}.
\end{align*}
Moreover, $\revKU{r}^{\text{HB}}_\square$ has minimal $\cZ'(\square)$--norm among all extensions in $\mathcal{E}(\reviewtwo{r};\cZ'(\square))$.
\end{lemma}

\begin{proof}
Since $P_\cU$ is bounded with $\|P_\cU\|=1$ and $P_\cU|_{\cU(\square)} = \mathrm{Id}$,
$\revKU{r}^{\text{HB}}_\square\in$ is a bounded linear functional on $\cZ(\square)$ that
coincides with $\reviewtwo{r}$ on $\cU(\square)$, i.e., $\revKU{r}^{\text{HB}}_\square\in \mathcal{E}(\reviewtwo{r};\cZ'(\square))$ (see \eqref{def:ext}). Furthermore,
\begin{align*}
    \|\revKU{r}^{\text{HB}}_\square\|_{\cZ'(\square)} =
\sup_{z\neq 0}\frac{|\reviewtwo{r}(P_\cU z)|}{\|z\|_{Z(\square)}}
\le
\sup_{z\neq 0}\frac{\|\reviewtwo{r}\|_{\cU'(\square)}\|P_\cU z\|_{\cZ(\square)}}{\|z\|_{\cZ(\square)}}
\le \|\reviewtwo{r}\|_{\cU'(\square)}.
\end{align*}
Taking the supremum over $z\in \cU(\square)\setminus\{0\}$ yields the reverse inequality. 

To prove minimality, let $\review{r_\square} \in \mathcal{E}(\reviewtwo{r};\cZ'(\square))$. By the Riesz \revKU{Representation Theorem}, there exists $w \in \cZ(\square)$ such that $\revKU{r}_\square(z) = (w,z)_{\cZ(\square)},\,\forall z \in \cZ(\square)$. Since $\revKU{r_\square}\equiv \reviewtwo{r}$ on $\cU(\square)$, the orthogonal projection $P_{\cU}w$ is the Riesz representative of $\reviewtwo{r}$ in $\cU(\square)$. Therefore, 
\begin{align*}
\|w\|_{\cZ(\square)}^2
=
\|P_{\cU}w\|_{\cZ(\square)}^2
+
\|(I - P_{\cU})w\|_{\cZ(\square)}^2
\ge
\|P_{\cU}w\|_{\cZ(\square)}^2.
\end{align*}
Thus,
\begin{align*}
\|\revKU{r}_\square\|_{\cZ'(\square)}=\|w\|_{\cZ(\square)}
\ge
\|P_{\cU}w\|_{\cZ(\square)}
=
\|\revKU{r}^{HB}_{\square}\|_{\cZ'(\square)},
\end{align*}
which proves minimality.
\end{proof}

\begin{remark}
Lemma~\ref{lem:HB-Hilbert-extension} shows that in Hilbert spaces the optimal (Hahn--Banach) extension is realized by \revKU{the} orthogonal projection. The \revKU{challenge in practice} is therefore not \revKU{the} existence but the computation of $P_\cU$ \revKU{without explicitly resolving} the geometry of $\Omega$. \revKU{Recall, that $\cU(\square)$ is isomorphic to $\cY(\Omega)$. Moreover, we have seen in the examples that $\cU(\square)$ is defined there by the image of $\cY(\Omega)$ under an extension operator. This means, that $\cU(\square)$ (and thus $P_\cU$) require to resolve the geometry of $\Omega$, which we want to avoid.}
\end{remark}

\revKU{
\begin{remark}\label{efficient implementation}
Concerning an efficient implementation, we shall always assume that we can decide whether $x\in\square$ belongs to $\Omega$ or to $\square\setminus\Omega$, \reviewNR{for instance} by \reviewNR{evaluating} an indicator function. \reviewNR{Thus, the geometry of $\Omega$ enter through pointwise membership queries, e.g.~at quadrature points.} This does not require a geometric resolution of $\Omega$ in terms of a triangulation \reviewNR{or any geometry conforming discretization}. Think, for instance, of $\Omega$ given in terms of a black and white image. This means that we assume to be able to perform quadrature on $\Omega$.
\end{remark}
}

\subsubsection{Penalized projection on $\cU(\square)$.}
We \revKU{describe an} approximation of the norm-preserving extension from Lemma~\ref{lem:HB-Hilbert-extension} \revKU{only requiring computations on $\square$}. The first step is to approximate the orthogonal projection onto $\cU(\square)\subset\cZ(\square)$ by solving a penalized variational problem posed on the enveloping domain $\square$\revKU{, which} avoids the explicit computation of $P_{\cU}$ and \reviewNR{the construction of a geometry-conforming discretization of} $\Omega$.

The variational problem involves a penalty term acting on the complement $\square\setminus\Omega$, which controls the distance to $\cU(\square)$. The required properties of this penalty are stated in the following assumption.

\begin{assumption}\label{penaltyterm}
Let $\cZ(\square)$ be a real Hilbert space with inner product $(\cdot,\cdot)_{\cZ(\square)}$ and let $\cU(\square)\subset \cZ(\square)$ be a closed subspace so that Assumption \ref{assume:Isomorph} holds. Assume that there exists a continuous, symmetric, positive semi-definite bilinear form 
$$
(\cdot,\cdot)_{\square\setminus\Omega} : \cZ(\square)\times \cZ(\square)\to \mathbb{R}
$$
such that
\revKU{
\begin{compactenum}[(i)]
\item $\cU(\square)=\{z\in \cZ(\square): (z,z)_{\square\setminus\Omega}=0\}$;
\item there exists a constant $C_{\mathrm{pen}}>0$ such that
\begin{equation}
\label{eq:pen-stability}
\|(I-P_\cU)z\|_{\cZ(\square)}\le C_{\mathrm{pen}}\,(z,z)_{\square\setminus\Omega}^{1/2}
\qquad \forall z\in \cZ(\square).
\end{equation}
\end{compactenum}}
\end{assumption}

\revKU{The first condition means that the bilinear form acts on the complement $\square\setminus\Omega$ and can be realized e.g.\ by integrating over the characteristic function $\chi_{\square\setminus\Omega}$. The second} condition \revKU{is} a coercivity property in $\square\setminus\Omega$, that is $(\cdot,\cdot)_{\square\setminus\Omega}$ controls the $\cZ(\square)$-distance to the constrained subspace $\cU(\square)$.

For $\lambda>0$, we \review{replace $P_\cU$ with the penalized projection operator $P_{\cU}^\lambda : \cZ(\square)\to \cZ(\square)$ defined by
$$
P_{\cU}^{\lambda}z
:=
\arg\min_{v\in\cZ(\square)}
\left(
\frac{1}{2}\|v-z\|_{\mathcal Z(\square)}^2
+
\frac{\lambda}{2}(v,v)_{\square\setminus\Omega}
\right),
$$ for $z\in\cZ(\square)$ fixed. D}efine the penalized bilinear form
\begin{equation}\label{penform}
a_\lambda(v,w) := (v,w)_{\cZ(\square)} + \lambda\, (v,w)_{\square\setminus\Omega},
\qquad v,w\in \cZ(\square)
\end{equation}
\review{Then, $P_{\cU}^{\lambda}z$ is, equivalently,} the solution of 
\begin{equation}\label{eq:projection}
a_\lambda(P_\cU^\lambda z, w) = (z,w)_{\cZ(\square)} \qquad \forall w\in Z(\square),
\end{equation}
\review{Thus, $P_{\cU}^{\lambda}$} provides a penalized approximation of the orthogonal projection $P_{\cU}$ onto $\cU(\square)$, as the penalty term enforces the constraint in the definition of $\cU(\square)$ asymptotically as $\lambda\to \infty$.

\begin{lemma}
\label{lem:penalized-projection}
For any $z\in \cZ(\square)$ the problem \eqref{eq:projection} admits a unique solution $P_\cU^\lambda z\in \cZ(\square)$. Moreover, $P_\cU^\lambda: \cZ(\square)\to \cZ(\square)$ is linear and \review{uniformly} bounded\revKU{, i.e., $\|P^\lambda_{\cU}\|_{\cL(\cZ(\square),\cZ(\square))}\le 1$.}
\end{lemma}
\begin{proof}
The bilinear form $a_\lambda(\cdot,\cdot)$ is continuous and coercive on
$\cZ(\square)$, since
$$a_\lambda(z,z) := \|z\|^2_{\cZ(\square)} + \lambda\, (z,z)_{\square\setminus\Omega}\ge \|z\|^2_{\cZ(\square)}$$
for any $z\in \cZ(\square)$, where we used that $(\cdot,\cdot)_{\square\setminus\Omega}$ is semi-positive definite.~The Lax-Milgram theorem implies uniqueness and continuous dependence of the solution from the right-hand side. The latter gives the stability estimate $\|P_\cU^\lambda z\|_{\cZ(\square)}
\le \|z\|_{\cZ(\square)}$. Hence $P_\cU^\lambda$ is bounded, and thus continuous. \review{In particular,
\begin{align*}
\|P^\lambda_{\cU}\|_{\cL(\cZ(\square),\cZ(\square))}:=\sup_{z\in\cZ(\square)\setminus{0}}\frac{\|P_\cU^\lambda z\|_{\cZ(\square)}}{\|z\|_{\cZ(\square)}}
\le 1,
\end{align*}
which completes the proof.}
\end{proof}
Next, we analyze the convergence of the penalized projection as $\lambda \to \infty$. 
\begin{proposition}
\label{prop:penalized-projection-convergence}
Let Assumption \ref{penaltyterm} hold, and let $z\in\cZ(\square)$ be fixed. Then,
$$
P_\cU^\lambda z \rightarrow P_\cU z \,\,\, \text{strongly in}\,\, \cZ(\square)\,\,\, \text{as } \lambda\to\infty.
$$
Moreover, the following estimates hold for all $\lambda>0$:
\begin{align}
\label{eq:leakage-est}
(P_\cU^\lambda z,P_\cU^\lambda z)^{1/2}_{\square\setminus\Omega} &\le \lambda^{-1/2}\,
\|z-P_\cU z\|_{\cZ(\square)},\\
\label{eq:Z-bound}
\|P_\cU^\lambda z-z\|_{\cZ(\square)} &\le \|P_\cU z-z\|_{\cZ(\square)}.
\end{align}
\end{proposition}

\begin{proof}
\emph{Step 1 (Variational characterization).}
For $\lambda>0$ define the strictly convex functional $J_\lambda:\cZ(\square)\to\mathbb{R}$ by
$$
J_\lambda(v):=\tfrac12\|v-z\|_{\cZ(\square)}^2+\tfrac{\lambda}{2}(v,v)_{\square\setminus\Omega}.
$$
A direct Euler--Lagrange computation shows that $u\in \cZ(\square)$ minimizes $J_\lambda$ if and only if
$$
(z-u,v)_{\cZ(\square)}=\lambda\, (u,v)_{\square\setminus\Omega} \qquad \forall v\in {\cZ(\square)},
$$
which is equivalent to \eqref{eq:projection}. Therefore, $u=P_\cU^\lambda z$ is the unique minimizer of
$J_\lambda$.

\emph{Step 2.}
Since $P_\cU z\in \cU(\square)$, it holds that
$J_\lambda(P_\cU z)=\frac12\|P_\cU z-z\|_{\cZ(\square)}^2$.
By minimality of $P_\cU^\lambda z$,
$$
\tfrac12\|P_\cU^\lambda z-z\|_{\cZ(\square)}^2+\tfrac{\lambda}{2}(P_\cU^\lambda z,P_\cU^\lambda z)_{\square\setminus\Omega}
=J_\lambda(P_\cU^\lambda z)\le J_\lambda(P_\cU z)=\tfrac12\|P_\cU z-z\|_{\cZ(\square)}^2,
$$
which implies \eqref{eq:leakage-est} by dropping the first term on the left.
The same inequality also yields \eqref{eq:Z-bound} by dropping the second term.

\emph{Step 3 (Weak convergence to $P_\cU z$).}
Estimate \eqref{eq:Z-bound} shows that $\{P_\cU^\lambda z\}_{\lambda>0}$ is bounded in $\cZ(\square)$. \revKU{Thus, for} any sequence $\lambda_n\to\infty$, there exists a subsequence, again denoted
by $\lambda_n$, and a $\bar u\in \cZ(\square)$ such that
$P_\cU^{\lambda_n} z\rightharpoonup \bar u$ weakly in $\cZ(\square)$.
By \eqref{eq:leakage-est} we have
$(P_\cU^{\lambda_n} z,P_\cU^{\lambda_n} z)_{\square\setminus\Omega}\to 0$. Since the map
$$
v \mapsto (v,v)_{\square\setminus\Omega}
$$
is convex and continuous, it is weakly lower semicontinuous on $\cZ(\square)$. Therefore,
$$
(\bar u,\bar u)_{\square\setminus\Omega}
\le
\liminf_{n\to\infty}
(P_\cU^{\lambda_n} z,P_\cU^{\lambda_n} z)_{\square\setminus\Omega}
=0.
$$
By nonnegativity of $(\cdot,\cdot)_{\square\setminus\Omega}$, this implies
$$
(\bar u,\bar u)_{\square\setminus\Omega} = 0
\qquad \text{on}\,\,\square\setminus\Omega,
$$
so $\bar u\in \cU(\square)$.
Moreover, for any $u\in \cU(\square)$ we have $J_{\lambda_n}(u)=\frac12\|u-z\|_{\cZ(\square)}^2$, and since $P_\cU^{\lambda_n}z$ minimizes $J_{\lambda_n}$,
$$
J_{\lambda_n}(P_\cU^{\lambda_n} z)\le J_{\lambda_n}(u)=\tfrac12\|u-z\|_{\cZ(\square)}^2.
$$
Dropping the nonnegative penalty term and taking $\liminf$ yields, by weak lower semicontinuity of
$v\mapsto\|v-z\|_{\cZ(\square)}^2$,
$$
\tfrac12\|\bar u-z\|_{\cZ(\square)}^2
\le \tfrac12\liminf_{n\to\infty}\|P_\cU^{\lambda_n} z-z\|_{\cZ(\square)}^2
\le \tfrac12\|u-z\|_{\cZ(\square)}^2 \qquad \forall u\in \cU(\square).
$$
Thus $\bar u$ minimizes $\|u-z\|_{\cZ(\square)}$ over $\cU(\square)$. By uniqueness of the orthogonal projection, $\bar u=P_\cU z$. Since every weakly convergent subsequence has the same limit, it follows
that
$$
P_\cU^\lambda z \rightharpoonup P_\cU z \quad \text{weakly in } {\cZ(\square)} \quad \text{as } \lambda\to\infty.
$$

\emph{Step 4 (Strong convergence).}
From \eqref{eq:Z-bound} we obtain $$\limsup_{\lambda\to\infty}\|P_\cU^\lambda z-z\|_{\cZ(\square)}\le \|P_\cU z-z\|_{\cZ(\square)}.$$
On the other hand, weak convergence yields the lower semicontinuity estimate
$$
\|P_\cU z-z\|_{\cZ(\square)} \le \liminf_{\lambda\to\infty}\|P_\cU^\lambda z-z\|_{\cZ(\square)}.
$$
Hence, $\|P_\cU^\lambda z-z\|_{\cZ(\square)}\to \|P_\cU z-z\|_{\cZ(\square)}$. Weak convergence together with convergence of norms in $\cZ(\square)$ yields strong convergence:
$$
\|P_\cU^\lambda z-P_\cU z\|_{\cZ(\square)} \to 0 \quad \text{as } \lambda\to\infty,
$$
\revKU{which completes the proof.}
\end{proof}
\subsubsection{Penalized extension}\label{subsec: Penalized extension} We recall that the restriction operator $\RBox:\cU(\square)\to \cY(\Omega)$ from Assumption \ref{assume:Isomorph} \revKU{needs to be} an isomorphism. \revKU{Then}, every functional $\revKU{r}_\Omega\in \cY'(\Omega)$ induces a functional $\reviewtwo{r}=\aRBox \revKU{r}_\Omega \in \cU'(\square)$ by 
$$
\langle \RBox' \revKU{r}_\Omega, u\rangle_{\cU(\square)} = \langle \revKU{r}_\Omega,\RBox u\rangle_{\cY(\Omega)},
\quad \forall u\in\cU(\square),\ \forall \revKU{r}_\Omega\in\cY'(\Omega).
$$
Equivalently $\forall z\in\cZ(\square)$, the (exact) norm preserving (cf.~with \eqref{eq:optextension}) extension of $\reviewtwo{r}$ \review{to $\cZ'(\square)$, i.e.~$\revKU{r}^{\text{HB}}_\square$,} satisfies
\begin{align}\label{exact}
\reviewmath{\revKU{r}^{\text{HB}}_\square(z)}
= \langle \reviewtwo{r}, P_\cU z\rangle_{\cU(\square)}
= \langle \revKU{r}_\Omega, \RBox(P_\cU z)\rangle_{\cY(\Omega)}.
\end{align}

Once the penalized projection is available, it would be convenient to define the penalized extension by replacing the exact projection $P_\cU$ in \eqref{exact} with its penalized counterpart  $P_\cU^\lambda$. However, note that $P_\cU^\lambda z \notin \revKU{\cU}(\square)$ \reviewtwo{in general}, so \review{the expression $\RBox(P_\cU^\lambda z)$ \revKU{analogous to \eqref{exact}} is not well-defined for arbitrary $z\in\cZ(\square)$.}. To overcome this, \review{consider a continuous extension of the restriction operator $\RBox$ such that\footnote{Such an extension is \revKU{in general} not unique.}
$$
\tRBox:\cZ(\square)\to\cY(\Omega),\,\,\,\textit{such that}\,\,\,\RBox = \tRBox|_{\cU(\square)}.
$$}%
\revKU{
Its adjoint $\atRBox:\cY'(\Omega)\to\cZ'(\square)$ is defined as usual by $\dual{\atRBox r_\Omega,z}_{\cZ(\square)}=\dual{r_\Omega,\tRBox z}_{\cY(\Omega)}$ for all $r_\Omega\in\cY'(\Omega)$ and $z\in\cZ(\square)$. Then, for any $r_\Omega\in\cY'(\Omega)$, the penalized extension $r_\square^\lambda$ is defined by
}
\revKU{
\begin{equation}\label{penalized_ext}
\revKU{r}^{\lambda}_\square(z)
:=
\langle \revKU{r}_\Omega, \tRBox(P_\cU^\lambda z)\rangle_{\cY(\Omega)},
\quad z\in\cZ(\square),
\end{equation}
which it is well-defined for every $\lambda>0$\review{, in the sense that $\reviewmath{\revKU{r}^{\lambda}_\square}|_{\cU(\square)}=\review{r}$.} This, in turn, follows from $P^\lambda_{\cU}u=u$ for $u\in\cU(\square)\reviewmath{\subset \cZ(\square)}$ and $\tRBox|_{\cU(\square)}=\RBox$. Apparently, $r_\square^\lambda = (P_\cU^\lambda)' \atRBox r_\Omega$, where $(P_\cU^\lambda)': \cZ'(\square)\to\cZ'(\square)$ is the adjoint of $P_\cU^\lambda$.
} 

\begin{corollary}\label{ext_convergence}
Let Assumption \ref{penaltyterm} hold\revKU{, }$z\in\cZ(\square)$ \revKU{fixed}, $\tRBox$ as above, and $P_\cU^\lambda z \to P_\cU z$ in $\cZ(\square)$ as $\lambda\to\infty$. Then, \review{for any $\revKU{r}_\Omega \in\cY'(\Omega)$} we have
$$
\reviewmath{\revKU{r}^{\lambda}_\square(z)
\revKU{= ((P_\cU^\lambda)' \atRBox r_\Omega) (z)}
\to \revKU{r}^{\text{HB}}_\square(z)},
\qquad z\in\cZ(\square).
$$
\end{corollary}
\review{\begin{proof}
By Proposition \ref{prop:penalized-projection-convergence}, $P^\lambda_{\cU}z\to P_{\cU}z$ in $\cZ(\square)$. Since $\tRBox\in \cL(\cZ(\square),\cY(\Omega))$, it follows that
$$
  \tRBox(P^\lambda_{\cU}z)\to \tRBox(P_{\cU}z)\qquad\text{in }\cY(\Omega).
$$
Moreover, $P_{\cU}z\in\cU(\square)$ implies $\tRBox(P_{\cU}z)=\RBox(P_{\cU}z)$. Applying the latter to \eqref{penalized_ext} we obtain
\begin{align*}
\reviewmath{\revKU{r}^{\lambda}_\square(z)}
&=
\langle
\revKU{r}_\Omega,
\tRBox(P^\lambda_{\cU}z)
\rangle_{\cY(\Omega)}
\to
\left\langle
\revKU{r}_\Omega,
\RBox(P_{\cU}z)
\right\rangle_{\cY(\Omega)}
=
\reviewmath{\revKU{r}^{\text{HB}}_\square(z)},
\end{align*}
which completes the proof.
\end{proof}}

\review{To quantify the norm deterioration caused by $\revKU{r}_\square^\lambda$ \revKU{for a given extension $r=\aRBox r_\Omega \in \cU'(\square)$}, we define its \emph{extension quality constant} as the maximal ratio of $\|r_\square^\lambda\|_{\cZ'(\square)}$ and $\|r\|_{\cU'(\square)}$. In order to do so, note, that 
\begin{align*}
r_\square^\lambda &= (P_\cU^\lambda)' \atRBox r_\Omega= (P_\cU^\lambda)'\atRBox(\aRBox)^{-1}r=: Q^\lambda_{\Omega,\square}r.
\end{align*}
Then, define the  extension quality constant as
\begin{equation}\label{constant_E}
C_\square^\lambda:=
\revKU{
\| Q^\lambda_{\Omega,\square}\|_{\cL(\cU'(\square),\cZ'(\square))}
=
}
\sup_{0\neq \review{r}\in \cU'(\square)} 
\frac{\|\revKU{(P_\cU^\lambda)'\atRBox(\aRBox)^{-1}r}\|_{\cZ'(\square)}}{\|\review{r}\|_{\cU'(\square)}} \ge 1.
\end{equation}
For the norm--preserving extension $\revKU{r}^{\text{HB}}_\square$ the corresponding extension quality constant equals $1$.~Note that $C_\square^\lambda$ is distinct from the isomorphism constant $C_\square$ in Lemma \ref{lemma:extBoundsBanachSpacev1}. 
The constant $C^\lambda_\square$ can be estimated starting from definition \eqref{penalized_ext},   
\begin{align*}
|\revKU{r}^{\lambda}_\square(z)|
&\le
\|\revKU{r}_\Omega\|_{\cY'(\Omega)}\|\tRBox\|_{\cL(\cZ(\square),\cY(\Omega))}\|P^\lambda_{\cU}z\|_{\cZ(\square)}
\\
&\le
C_\square\|\review{r}\|_{\cU'(\square)}\|\tRBox\|_{\cL(\cZ(\square),\cY(\Omega))}
\|P^\lambda_{\cU}\|_{\cL(\cZ(\square),\cZ(\square))}\|z\|_{\cZ(\square)}.
\end{align*}
Here we used continuity of $\tRBox$ in the first step, and the norm-equivalence from Lemma \ref{lemma:extBoundsBanachSpacev1} in the second. Moreover, the penalized projection $P^\lambda_{\cU}$ is uniformly bounded (see Lemma \ref{penform}). Therefore, taking the supremum over $z\in\cZ(\square)\setminus\{0\}$ and rearranging, we get the bound
\begin{equation}\label{constant}
\|\revKU{r}^{\lambda}_\square\|_{\cZ'(\square)}\le C_\square\|\tRBox\|_{\cL(\cZ(\square),\cY(\Omega))}\|\review{r}\|_{\cU'(\square)}.
\end{equation}
On the other hand,
\begin{align*}
\|\revKU{r}^{\lambda}_\square\|_{\cZ'(\square)}
=
\sup_{z\in \cZ(\square)\setminus\{0\}}
\frac{|\revKU{r}^{\lambda}_\square(z)|}{\|z\|_{\cZ(\square)}}
&\ge
\sup_{u\in \cU(\square)\setminus\{0\}}
\frac{|\revKU{r}^{\lambda}_\square(u)|}{\|u\|_{\cZ(\square)}}
=\\
&=\sup_{u\in \cU(\square)\setminus\{0\}}
\frac{|\review{r}(u)|}{\|u\|_{\cZ(\square)}}
=
\|\review{r}\|_{\cU'(\square)},
\end{align*}
where the crucial last equality is an immediate consequence of the extension property $\revKU{r}^{\lambda}_\square|_{\cU(\square)}=\review{r}$. 
Combining the last two inequalities, 
\begin{equation}\label{constant_E_lambda}
  1\le C_\square^\lambda
  \le
  C_\square
  \|\tRBox\|_{\cL(\cZ(\square),\cY(\Omega))}.
\end{equation}
Observe that the deterioration constant $C_\square^\lambda$ is bounded independently of $\lambda$.
}

\subsubsection{\revKU{Example: elliptic PDEs,  $\cY(\Omega)=H^{\reviewmath{1}}_0(\Omega)$}}
\label{subsec: computable Penalized extension}
Building on the results of \S\ref{Sec:Extension}, we now specialize the abstract construction to the case
$$
\cY(\Omega)=H^{\reviewmath{1}}_0(\Omega), \,\,\cZ(\square)=H^{\reviewmath{1}}_0(\square)\reviewmath{=\cY(\square)},
$$
equipped with \revKU{the} inner product $\revKU{(z,w)_{\cZ(\square)} := (\nabla z,\nabla w)_{L_2(\square)}}$ inducing the energy norm on $H^{\reviewmath{1}}_0(\square)$. \revKU{In order to satisfy Assumption \ref{penaltyterm}, we choose
\begin{align*}
    (v,w)_{\square\setminus\Omega} := (\nabla v, \nabla w)_{L_2(\square\setminus\Omega)},
    \quad v,w\in\cZ(\square)=H^1_0(\square),
\end{align*}
and define (in view of Assumption \ref{penaltyterm}(i))
\begin{align*}
    \cU(\square)
    &:=\{z\in H^1_0(\square): (\nabla z,\nabla z)_{\square\setminus\Omega}=0\}
    =\{z\in H^1_0(\square): \nabla z|_{\square\setminus\Omega}=0 \text{ a.e.}\}.
\end{align*}
The property $\nabla z|_{\square\setminus\Omega}=0$ a.e.~ together with the homogeneous boundary condition on $\partial\square$ yields $z = 0$ a.e.\ in $\square \setminus \Omega$. Thus, $\cU(\square)$ coincides with the space of zero extensions of $H^1_0(\Omega)$-functions, as defined in \eqref{eq:UdefSobolev}.
In order to verify Assumption \ref{penaltyterm}(ii), we consider the space 
\begin{align*}
    H^1_{\partial\square}(\square\setminus\Omega):=\{z\in H^1(\square\setminus\Omega):z|_{\partial\square}=0\},
\end{align*} 
equipped with the norm $\|\cdot\|_{H^1_{\partial\square}(\square\setminus\Omega)}=(\cdot,\cdot)_{\square\setminus\Omega}^{1/2}$, and the restriction operator 
\begin{align*}
    R_{\square\setminus\Omega\leftarrow\square}: H^1_0(\square)\to H^1_{\partial\square}(\square\setminus\Omega), 
    \quad 
    \text{ defined as }
    \quad
    R_{\square\setminus\Omega\leftarrow\square}\, z
    := z|_{\square\setminus\Omega}.
\end{align*}
Note, that $R_{\square\setminus\Omega\leftarrow\square}$ is not injective, and its kernel is precisely the zero-extension subspace $\cU(\square)$, i.e., $\cU(\square)=\mathop{ker} (R_{\square\setminus\Omega\leftarrow\square})$.
In addition, we choose a continuous linear extension $E_{\square\setminus\Omega\rightarrow\square}:H^1_{\partial\square}(\square\setminus\Omega)\to H^1_0(\square)$, such that $R_{\square\setminus\Omega\leftarrow\square}\circ E_{\square\setminus\Omega\rightarrow\square}=\id_{H^1_{\partial\square}(\square\setminus\Omega)}$ (see Remark \ref{harmonic} below).\footnote{\review{Zero-extension is not appropriate in this case, since a function in $H^1(\square\setminus\Omega)$ need not vanish on $\partial\Omega$. Instead, we can use a bounded extension operator for instance obtained by extending the trace on $\partial\Omega$ harmonically into $\Omega$ and keeping the original function on $\square\setminus\Omega$ (see Remark \ref{harmonic}).}} 
By continuity, defining the constant $C_{\square\setminus\Omega}:=\|E_{\square\setminus\Omega\rightarrow\square}\|_{\cL\left(H^1_{\partial\square}(\square\setminus\Omega),H^1_0(\square)\right)}$, yields
\begin{align}\label{E_cont}
    \|E_{\square\setminus\Omega\rightarrow\square}z\|_{H^1_0{(\square)}}
    \le  C_{\square\setminus\Omega}\, \|z\|_{H^1_{\partial\square}(\square\setminus\Omega)}
    = C_{\square\setminus\Omega}\, (z,z)_{\square\setminus\Omega}^{1/2}. 
\end{align}
The idea is to subtract from $z\in H^1_0(\square)$ a function with the same restriction to $\square\setminus\Omega$, i.e., for $z\in H^1_0(\square)$, we set
\begin{align*}
    u_z := z-E_{\square\setminus\Omega\rightarrow\square} R_{\square\setminus\Omega\leftarrow\square}\, z.
\end{align*}
Then, $R_{\square\setminus\Omega\leftarrow\square}\, u_z=0$, thus, $u_z\in\cU(\square)$.\footnote{\review{In general, $E_{\square\setminus\Omega\rightarrow\square} R_{\square\setminus\Omega\leftarrow\square}z\neq z$, but it is an extension of $R_{\square\setminus\Omega\leftarrow\square}z\in H^1_{\partial\square}(\square\setminus\Omega)$ to $H^1_0{(\square)}$.}} Since $P_\cU z$ is the best approximation to $z$ from $\cU(\square)$, we obtain
\begin{align*}
    \|(I-P_\cU)z\|_{H^1_0(\square)}
    &= \inf_{u\in\cU(\square)}\|z-u\|_{H^1_0(\square)}
    \le \|z-u_z\|_{H^1_0(\square)}\\
    &=\|E_{\square\setminus\Omega\rightarrow\square} R_{\square\setminus\Omega\leftarrow\square}\, z\|_{H^1_0(\square)}\\
    &\le C_{\square\setminus\Omega} \|R_{\square\setminus\Omega\leftarrow\square}\, z\|_{H^1_{\partial\square}(\square\setminus\Omega)}
    = C_{\square\setminus\Omega}\, (z,z)_{\square\setminus\Omega}^{1/2}.
\end{align*}
Hence, Assumption \ref{penaltyterm}(ii) holds with constant $C_{\text{pen}}=C_{\square\setminus\Omega}$. 

\begin{remark}[Harmonic extension]\label{harmonic}
    We bound the constant $C_{\square\setminus\Omega}$ for a particular choice of $E_{\square\setminus\Omega\rightarrow\square}$: For $w\in H^1_{\partial\square}(\square\setminus\Omega)$, we define a \emph{harmonic extension} as 
    \begin{align*}
        E_{\square\setminus\Omega\to\square}\, w
        :=
          \begin{cases}
                w, & \text{in } \square\setminus\Omega,\\
                h_w, & \text{in } \Omega,
          \end{cases}
    \end{align*}
where $h_w\in H^1(\Omega)$ is harmonic in $\Omega$ and satisfies $h_w|_{\partial\Omega}=w|_{\partial\Omega}.$ Obviously,
$\|E_{\square\setminus\Omega\to\square}w\|_{H^1_0(\square)}^2
  = \|\nabla w\|_{L_2(\square\setminus\Omega)}^2
  + \|\nabla h_w\|_{L_2(\Omega)}^2$. 
By the stability of the harmonic extension, there exists a constant $C_{\rm harm}>0$ such that the estimate $\|\nabla h_w\|_{L_2(\Omega)} \le C_{\rm{harm}}\, \|w|_{\partial\Omega}\|_{H^{1/2}(\partial\Omega)}$ holds. Applying a trace estimate on $\square\setminus\Omega$ yields the bound  $\|w|_{\partial\Omega}\|_{H^{1/2}(\partial\Omega)} \le C_{\rm{tr}}\,\|w\|_{H^1(\square\setminus\Omega)}$. Since $w$ vanishes on the outer boundary $\partial\square$, a Poincar\'e inequality on $\square\setminus\Omega$ gives $\|w\|_{H^1(\square\setminus\Omega)} \le C_{\rm{P}}\,\|\nabla w\|_{L_2(\square\setminus\Omega)}$, where $C_{\rm{P}}$ is the Poincar\'e constant associated to the domain  $\square\setminus\Omega$. Combining the last three estimates yields
$\|\nabla h_w\|_{L_2(\Omega)}
  \le C_{\rm{harm}}\, C_{\rm{tr}}\, C_{\rm{P}}\,
  \|\nabla w\|_{L_2(\square\setminus\Omega)}$ and finally 
\begin{align*}
  \|E_{\square\setminus\Omega\to\square}w\|_{H^1_0(\square)}^2
  &= \|\nabla w\|_{L_2(\square\setminus\Omega)}^2
    + \|\nabla h_w\|_{L_2(\Omega)}^2\\
  &\le
  \left(
    1+C_{\rm{harm}}^2\, C_{\rm{tr}}^2\, C_{\rm{P}}^2
  \right)
    \|\nabla w\|_{L_2(\square\setminus\Omega)}^2
    =: C_{\square\setminus\Omega}^2  \|\nabla w\|_{L_2(\square\setminus\Omega)}^2,
\end{align*}
i.e., \eqref{E_cont} and the constant only depends on the geometry of $\Omega$, $\square$ and $\square\setminus\Omega$ through the trace, harmonic extension, and Poincar\'e constants.
\end{remark}
}

Now, given $z\in \revKU{H^1_0(\square)}$, the penalized projection $P_\cU^\lambda z \in \cZ(\square)$ \revKU{is defined by
\begin{align*}
    a_\lambda(P_\cU^\lambda z,w)
    &= (\nabla P_\cU^\lambda z, \nabla w)_{L_2(\square)} + \lambda\, (\nabla P_\cU^\lambda z, \nabla w)_{L_2(\square\setminus\Omega)}
    = (\nabla z, \nabla w)_{L_2(\square)}\\
    &= (z,w)_{\cZ(\square)}
\end{align*}
for all $w\in H^1_0(\square)$. By minimizing $J_\lambda$ on $\cZ(\square)$, the penalty term suppresses the energy of $P_\cU^\lambda z$ on $\square\setminus\Omega$, so that $P_\cU^\lambda z$ approximates the orthogonal projection onto $\cU(\square)$ as $\lambda\to\infty$.}

Let \review{$r_\Omega \in \cY'(\Omega)=(H^{1}_0(\Omega))'=H^{-1}(\Omega)$} be \revKU{given}. \review{\revKU{Then, consider some} continuous extension\footnote{\review{The choice of $\bar{r}_\Omega\in (H^{1}(\Omega))'$ is not unique. For instance, if $\bar{r}_\Omega$ is one extension, then for any $g\in H^{-1/2}(\partial\Omega)$ the functional $v\mapsto \bar{r}_\Omega(v) + \langle g,v|_{\partial\Omega}\rangle_{H^{1/2}(\partial\Omega)}$ on $H^1(\Omega)$ defines another extension, since the trace vanishes on $H^1_0(\Omega)$.}}
$$\bar{r}_\Omega\in (H^{1}(\Omega))', \quad \bar{r}_\Omega|_{H^1_0(\Omega)}=r_\Omega.$$
Aiming for a simple implementation, we define the restriction operator,
\begin{equation}\label{def:restriction}
\tRBox: H^1_0(\square)\to H^1(\Omega), \quad \tRBox z := z|_\Omega.
\end{equation}
We stress that $\tRBox z$ need not belong to $H^1_0{(\Omega)}$, for $z\in H^1_0(\square)$.}

Then, motivated by \eqref{penalized_ext}\review{, a practical} penalized extension of \review{$r_\square=\aRBox r_\Omega\in\cU'(\square)$ can be written as
\begin{equation}\label{eq:computablePenResidual}
    r_\square^\lambda(z):=\langle \bar{r}_\Omega, (P_\cU^\lambda z)|_\Omega\rangle_{H^{1}(\Omega)},
\quad z\in H^1_0(\square)
\end{equation}
In fact, this expression is well-defined, since
$$P_\cU^\lambda z \in H^{1}_0(\square) 
\revKU{\quad\text{implies that}\quad}
\tRBox(P_\cU^\lambda z)\in H^{1}(\Omega).$$
Moreover, if $z\in\cU(\square)$ then 
$$\tRBox(z)=\RBox(z)\in H^{1}_0(\Omega)$$
and, since in this case $P_\cU^\lambda z=z\in\cU(\square)$, \eqref{eq:computablePenResidual} yields
\begin{align*}
r_\square^\lambda(z)&=\langle \bar{r}_\Omega,z|_\Omega\rangle_{H^{1}(\Omega)}\\
&=\langle r_\Omega,z|_\Omega\rangle_{H^{1}_0(\Omega)}\qquad\,\,(z|_\Omega\in H^1_0(\Omega),\,\,\bar{r}_\Omega|_{H^1_0(\Omega)}=r_\Omega)\\
&=\langle \aRBox r_\Omega,z\rangle_{\cU(\square)} \quad(z|_\Omega=\RBox(z)\,\,\textit{and by definition of}\,\,\aRBox)
\end{align*}
This validates that $r^\lambda_\square\in\cZ'(\square)=(H^{1}_0(\square))'$, defined in \eqref{eq:computablePenResidual}}, coincides with $\aRBox r_\Omega$ on $\cU(\square)$, therefore, it is an extension of the latter.
\begin{remark}\label{Practical computation}
The operator $P_\cU^\lambda$ is obtained by solving a standard elliptic problem on the simple domain $\square$, with the penalty term being realised as a volume integral on $\square\setminus\Omega$\reviewNR{.}  \reviewNR{Using the pointwise membership information assumed in Remark \ref{efficient implementation}, this integral can be evaluated on a background discretization of $\square$ by evaluating $\chi_{\square\setminus\Omega}$ at the quadrature points. Similarly, t}he duality pairing \review{$\langle \bar{r}_\Omega, (\cdot)|_\Omega \rangle_{H^1(\Omega)}$} in \eqref{eq:computablePenResidual} is given by integrals over $\Omega$, so that it can be evaluated on $P_\cU^\lambda z\in \cZ(\square)$ without explicit evaluation of $\tRBox$, only by omitting values outside $\Omega$.
Hence, the construction avoids \reviewNR{the need for a geometry-conforming discretization of} $\Omega$, while providing a tunable, computable approximation of the norm--preserving extension.
\end{remark}
\review{\begin{remark}
In the elliptic case considered here, the estimate \eqref{constant_E_lambda} can be made explicit. Since $\cU(\square)$ is the zero-extension image of $H^1_0(\Omega)$, the extension operator $\EBox: H_0^{1}(\Omega) \rightarrow \cU(\square)$ is an isometry. Hence, $C_\square=1$ (cf. \eqref{eq:Isometry}). Moreover, for $\tRBox(z)=z|_\Omega$ as defined earlier in \eqref{def:restriction}, the Poincar\'{e} inequality gives
$$
\|\tRBox(z)\|_{H^1(\Omega)}\leq (1+C^2_{\text{P}})^{1/2}\|z\|_{H^1_0{(\square)}},
$$
where $C_{\text{P}}$ denotes the Poincar\'{e} constant on $\square$. 
Consequently, the estimate $\eqref{constant_E_lambda}$ in this case reads as
\begin{equation}\label{est:pen ext}
1\leq C_\square^\lambda\le (1+C^2_{\text{P}})^{1/2}
\end{equation}
and the deterioration constant is bounded independently of $\lambda$. 
\end{remark}}
Our final objective is to compute the Riesz representative $\hat r_\square^\lambda\in H^1_0(\square)$ of $r_\square^\lambda$, satisfying
\begin{equation}\label{riesz}
(\nabla {\hat r}_\square^\lambda,\nabla z)_{\revKU{L_2}(\square)}
=
r_\square^\lambda(z)
\qquad \forall z\in H^1_0(\square).
\end{equation}
We \review{now} show that the Riesz representative \review{of the penalized extension $r_\square^\lambda$} \review{is} equivalently characterized as the unique solution of a \revKU{simple} variational problem, given as
\begin{equation}\label{riesz_one}
(\nabla \hat r_\square^\lambda,\nabla v)_{\revKU{L_2}(\square)}
+
\lambda (\nabla \hat r_\square^\lambda,\nabla v)_{\revKU{L_2}(\square\setminus\Omega)}
=
\reviewmath{\langle \bar{r}_\Omega,v|_\Omega\rangle_{H^{1}(\Omega)}}
\qquad \forall v\in H^1_0(\square).
\end{equation}
\review{L}et $\reviewmath{w}\in H^1_0(\square)$ denote the unique solution of
$$
a_\lambda(\reviewmath{w},v)=\reviewmath{\langle \bar{r}_\Omega,v|_\Omega\rangle_{H^{1}(\Omega)}}
\qquad \forall v\in H^1_0(\square).
$$
The problem is well-posed by Lax--Milgram, since $a_\lambda$ is continuous and coercive on $H^1_0(\square)$, and the right-hand side is continuous \review{because} $\bar{r}_\Omega\in (H^1(\Omega))'$ \review{and the restriction $H^1_0(\square)\to H^1(\Omega)$ is continuous}.

We claim that $\reviewmath{w}$ is the Riesz representative of the penalized extension  $r_\square^\lambda$.
Indeed, testing the above equation with $v=P_\cU^\lambda z$, for any $z\in H^1_0(\square)$, yields
$$
a_\lambda(\reviewmath{w},P_\cU^\lambda z)
=
\reviewmath{\langle \bar{r}_\Omega,(P_\cU^\lambda z)|_\Omega\rangle_{H^{1}(\Omega)}}
=r_\square^\lambda(z).$$ 
Furthermore, testing equation \eqref{eq:projection} with $\reviewmath{w}$, we obtain
$$
a_\lambda(P_\cU^\lambda z,\reviewmath{w})
=
(z,\reviewmath{w})_{\cZ(\square)}
=
(\nabla z,\nabla \reviewmath{w})_{\revKU{L_2}(\square)}.
$$
\review{Since $a_\lambda$ is symmetric,} \review{t}he last two equations show that
$$(\nabla \reviewmath{w},\nabla z)_{\revKU{L_2}(\square)}=r_\square^\lambda(z)\qquad \forall z\in H^1_0(\square).
$$
\review{Therefore}, $\reviewmath{w}$ solves \eqref{riesz}\smallskip.

\review{\begin{remark}\label{exponential}
The penalized Riesz problem~\eqref{riesz_one} is an elliptic interface problem on $\square$. The coefficient $1+\lambda\chi_{\square\setminus\Omega}$ is constant on $\Omega$ and on $\square\setminus\overline{\Omega}$, but discontinuous across $\partial\Omega$. Consequently, the Riesz representative is generally only piecewise smooth\reviewNR{.} \reviewNR{The coefficient jump may affect the spatial approximation, while} the conditioning of the discrete linear system may deteriorate as $\lambda$ increases, unless a suitable preconditioner is employed. \reviewNR{At the same time, posing the auxiliary problem on the geometrically simple domain $\square$, structured discretizations and efficient solvers for elliptic problems remain available.} 

\reviewNR{We stress, however, that \eqref{riesz_one} is an auxiliary problem used to evaluate an a posteriori error estimator, rather than a discretization of the original PDE on $\Omega$. Its numerical solution therefore only needs to provide a sufficiently accurate approximation of the corresponding dual norm. The sensitivity with respect to the quadrature of the discontinuous coefficient and to the spatial resolution of the background discretization of the $\square$ domain is investigated numerically in Section~\ref{parametric domain}.} 
\end{remark}

\begin{remark}[Mixed boundary conditions]\label{mixed}
The \revKU{presented} framework is not restricted to  homogeneous Dirichlet problems. If $\partial\Omega=\Gamma_D\cup\Gamma_N$, with Dirichlet data imposed essentially on $\Gamma_D$ and Neumann data imposed variationally on $\Gamma_N$, one may replace $H^1_0(\Omega)$ by
$H^1_{\Gamma_D}(\Omega):= \{v\in H^1(\Omega): v|_{\Gamma_D}=0\}$.
The residual then belongs to $\cY'(\Omega)=(H^1_{\Gamma_D}(\Omega))'$, and contains a Neumann boundary functional, with $\dual{\cdot,\cdot}_{\Gamma_N}$ denoting the corresponding duality pairing on $\Gamma_N$, e.g.
$$
  \dual{r_\Omega(u^\theta),v}_{H^1_{\Gamma_D}(\Omega)}
  =
  (f,v)_{L_2(\Omega)}
  +
  \langle g_N,v\rangle_{\Gamma_N}
  -
  a(u^\theta,v),
  \qquad v\in H^1_{\Gamma_D}(\Omega)
$$
where $a(\cdot,\cdot)$ denotes the variational bilinear form associated with the PDE operator.  The extension/restriction construction can be done with $\cY(\Omega)=H^1_{\Gamma_D}(\Omega)$, provided a bounded extension operator $\EBox:H^1_{\Gamma_D}(\Omega)\to \cZ(\square)$ is chosen such that the associated restriction map is a bounded inverse on its range $\cU(\square):=\EBox H^1_{\Gamma_D}(\Omega)$. With this choice, Assumption~\ref{assume:Isomorph} is satisfied. Such extension operators are available, for instance, when $\Omega$ is a bounded domain with smooth boundary, see \cite[Chapter~8]{LionsMagenes1972}.
\end{remark}}

\subsubsection{\revKU{Example: parabolic PDEs, Bochner spaces, $\cY(\Omega)=L_2(I; H^{\reviewmath{1}}_0(\Omega))$}}
\label{subsec:computable-space-time}
For space-time formulations, all \revKU{above} developments carry over analogously to Bochner spaces using the space-time extension and restriction operators developed in \S\ref{Subsec:Bochner}. The construction extends to $\cY(\Omega) = \revKU{L_2}(I;H^1_0(\Omega))$ and $\cZ(\square) = \revKU{L_2}(I;H^1_0(\square))$, with the penalized projection being applied pointwise in time, i.e.,
$$
(P_\cU^\lambda v)(t) := P_\cU^\lambda(v(t)) \quad \text{for}\,\,t\in I\,\,\text{a.e.}.
$$  
Then, for instance, the \revKU{analogue} of \eqref{riesz_one} reads as 
\begin{equation*}
\int_I (\nabla \hat r_{\square}^\lambda(t),\nabla v(t))_{\revKU{L_2}(\square)}\,dt
+
\lambda\kern-3pt \int_I (\nabla \hat r_{\square}^\lambda(t),\nabla v(t))_{\revKU{L_2}(\square\setminus\Omega)}\,dt
\kern-1pt=\kern-4pt
\int_I\kern-1pt\reviewmath{\langle \bar{r}_{\Omega}(t),v(t)|_\Omega\rangle_{H^1(\Omega)}} \,dt
\end{equation*}
for all $v\in \revKU{L_2}(I;H^1_0(\square))$.

\section{Numerical experiments} \label{sec:NumResults}
We investigate the quantitative performance of the proposed error estimators for linear elliptic and parabolic PDEs. In order to do so, we compare the quantities listed in the following table.
\begin{center}
\includegraphics{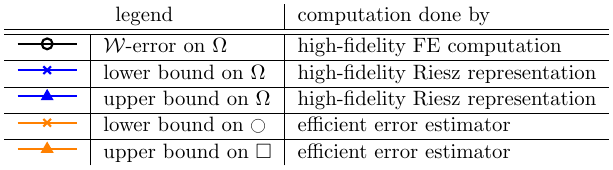}
\end{center}

The \enquote{true} $\mathcal{W}$-error (black) is computed \revKU{in $\cW(\Omega)$} by comparing the NN solution with a high-fidelity finite element reference solution. The upper and lower bounds in $\Omega$ (blue lines) are determined by using the Riesz representation w.r.t.\ a fine discretization of the PDE on the domain $\Omega$ and computing the dual norm of the residual by determining the primal norm of the Riesz representation. For all presented examples, we derive analytical estimates for the involved constants. Of course, this could also be replaced by solving corresponding generalized eigenvalue problems. These \enquote{Riesz estimates} on $\Omega$ are used as a reference only, as those bounds require high computational cost. It is clear that our error estimates cannot be better than the Riesz bounds on $\Omega$.

In order to compute our error estimator (orange), we need to solve \reviewNR{auxiliary Riesz problems} on $\mycirc$ and $\square$. Hence, we choose these domains in such a way that their geometry is rather simple, e.g.\ a circle or a rectangle allowing for highly efficient tensor product discretizations in spherical and canonical coordinates, respectively. For \reviewNR{computations on $\mycirc$}, we used highly efficient and accurate spectral methods, \cite{Canuto2006}. \reviewNR{The (penalized) Riesz problems on $\square$ are approximated by finite elements on fixed background meshes.} All experiments have been carried out in Python with FEniCSx\,\cite{Baratta2023} and PyTorch\,\cite{NEURIPS2019_9015}. Our code can be found on a \texttt{git} repository, \cite{codeGIT}.

\subsection{Dependence of the domain size}
We start by a problem in space only, i.e., an elliptic problem posed on a rather simple domain $\Omega := \lbrace (x,y)^T \in \mathbb{R}^2: x^2 + y^2 < 0.5^2 \rbrace$ being the circle shown in Figure\,\ref{fig:circleDomain}. Our goal is to investigate numerically how the upper and lower bound depend on the size of the domain. Moreover, we would like to know how the penalty changes the behaviour of the upper bound with respect to the domain size. In order to do so, we choose a parametric domain $(-\ell, \ell)^2$ \review{of} side length $2\ell \in (0,\infty)$\review{, and distinguish notation according to the role of the domain in the theoretical framework: for $0<\ell<\sqrt{2}/4$, we denote $\mycirc_\ell:=(-\ell,\ell)^2\subset\Omega$, whereas, for $\ell>1/2$, we set $\square_\ell:=(-\ell,\ell)^2\supset\Omega$.} In Figure\,\ref{fig:circleDomain}, we have pictured the two cases for some parameters $\ell_1$ and $\ell_2$. Furthermore, for $\sqrt{2}/4< \ell < 1/2$ \review{neither inclusion required by} the theory derived in the last section hold and thus, one can not draw any information from the bounds for those values of $\ell$. 
\begin{figure}[!htb]
	\begin{minipage}{1.0\textwidth}
		\centering
		
\includegraphics{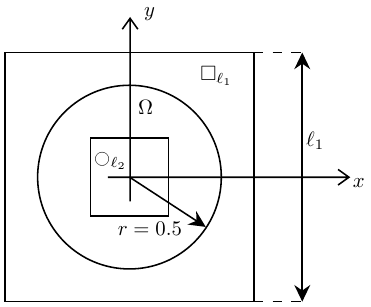}
\end{minipage}
	\caption{The circular domain $\Omega$ with radius $1/2$, an inner square domain denoted by $\mycirc_{\ell_2}$, and an outer square domain $\square_{\ell_1}$, with side lengths $\ell_2$ and $\ell_1$, respectively. Apparently, it holds that $\mycirc_{\ell_2}\subset\Omega\subset\square_{\ell_1}$.}
    \label{fig:circleDomain}
\end{figure}

The underlying PDE for this experiment is Poisson's problem
\begin{equation*}
    - \Delta u = f, \quad \text{in } \Omega
\end{equation*}
with zero boundary conditions and the $L_2$-function
\begin{equation*}
    f(x,y) = 
    \begin{cases}
        \sin(xy) + 10, & \text{if } y < 0.5,\\
        \vert \cos(xy) \vert + 20, & \text{if } y \ge 0.5.
    \end{cases}
\end{equation*}

The standard variational form of the PDE uses the trial and test space $\mathcal{W}\reviewmath{(\Omega)}=\mathcal{Y}\reviewmath{(\Omega)}=H^1_{0}(\Omega)$ yielding a Galerkin discretization and amounts finding  $u \in H^1_{0}(\Omega)$ such that 
\begin{equation*}
    \dual{r_\Omega(u),v}_{H^1_0(\Omega)} := (\nabla u, \nabla v )_{L_2(\Omega)} - (f, v)_{L_2(\Omega)} = 0 \quad \forall v \in H^1_{0}(\Omega).
	\end{equation*}

The approximation is \revKU{chosen} to be a PINN approximation, trained with the loss function $\mathcal{L}(\theta):= \sum_{(x,y) \in \mathcal{S}_\Omega} \vert (\Delta \reviewmath{u^\theta})(x,y) + f(x,y)\vert^2$, where the training set $\mathcal{S}_\Omega \subset \Omega$ consists of $2^{12}$ randomly choosen points. The NN architecture is given by the neuron vector $N = (2, 32, 32, 1)$, i.e. a rather small network. 

\review{
To apply the (penalized) extension construction, we first extend the residual
$r_\Omega(u^\theta)\in (H^1_0(\Omega))'=H^{-1}(\Omega)$ to a functional $\bar{r}_\Omega(u^\theta)\in (H^1(\Omega))'$ by defining
\begin{equation*}
    \langle \bar{r}_\Omega(u^\theta),v\rangle_{H^1(\Omega)}:=
    (\nabla u^\theta,\nabla v)_{\revKU{L_2}(\Omega)}-(f,v)_{\revKU{L_2}(\Omega)}\quad \forall v\in H^1(\Omega).
\end{equation*}
Clearly, it holds that $\bar{r}_\Omega(u^\theta)|_{H^1_0(\Omega)}=r_\Omega(u^\theta)$.}
\review{For $\ell>1/2$, so that $\Omega\subset\square_\ell$}, we define
\begin{equation*}
	\dual{r_{\square_\ell}(\reviewmath{u^\theta}), v }_{H^1_0(\square_\ell)} 
    := \dual{\reviewmath{\bar{r}}_\Omega(\reviewmath{u^\theta}), \reviewmath{v|_\Omega}}_{H^1(\reviewmath{\Omega})}, 
    \forall v \in H^1_{0}(\square_\ell).
\end{equation*}
\review{Then $r_{\square_\ell}(u^\theta)\in (H^1_0(\square_\ell))'= H^{-1}(\square_\ell)=:\mathcal{Z}'(\square_\ell)$, since the restriction map $v\mapsto v|_\Omega$ is continuous from $H^1_0(\square_\ell)$ to $H^1(\Omega)$}
 Furthermore, $r_{\square_\ell}(\reviewmath{u^\theta})$ is an extension of $R'_{\Omega \rightarrow \square_\ell} r_\Omega(\reviewmath{u^\theta})$, because $r_{\square_\ell}(\reviewmath{u^\theta}) \equiv R'_{\Omega \rightarrow \square_\ell} r_\Omega(\reviewmath{u^\theta})$ on $\mathcal{U}(\square_\ell)$, defined in \eqref{eq:UdefSobolev}. \review{Indeed, if $v\in\cU(\square_\ell)$, then $v|_\Omega\in H^1_0(\Omega)$, and therefore
$$
\langle r_{\square_\ell}(u^\theta),v\rangle_{H^1_0(\square_\ell)}
=
\langle r_\Omega(u^\theta),v|_\Omega\rangle_{H^1_0(\Omega)}.
$$}The penalized extension \review{$r_{\square_\ell}^\lambda(u^\theta)$} is \review{then} defined as in \eqref{eq:computablePenResidual}. 

For $\ell < \sqrt{2}/4$ and with the space $\mathcal{V}(\Omega)$ defined in \eqref{eq:VdefSobolev} we define \review{$$\langle r_{\reviewmath{\mycirc}_\ell}(u^\theta),v\rangle_{H^1_0(\mycirc_\ell)}
:=\langle r_\Omega(u^\theta),E_{\mycirc_\ell\to\Omega}v\rangle_{H^1_0(\Omega)},
\quad v\in H^1_0(\mycirc_\ell).$$}%
\review{By construction, both $r_{\square_\ell}(u^\theta)$ and $r_{\square_\ell}^\lambda(u^\theta)$ are extensions of $R'_{\Omega\rightarrow\square_\ell} r_\Omega (u^\theta)$, thus belong to $\mathcal{E}(R'_{\Omega\rightarrow\square_\ell} r_\Omega;H^{-1}(\square_\ell))$.} \review{Hence}, Corollary\,\ref{cor:boundsSobolev} applies and with the unity constants we get 
\begin{equation*}
	\Vert r_{\reviewmath{\mycirc}_\ell}(\reviewmath{u^\theta}) \Vert_{H^{-1}(\reviewmath{\mycirc}_\ell)} 
    	\kern-3pt\overset{\ell < \sqrt{2}/4}{\le} \Vert \nabla u - \nabla \reviewmath{u^\theta} \Vert_{L_2(\Omega)} 
	= \Vert r_{\Omega}(\reviewmath{u^\theta}) \Vert_{H^{-1}(\Omega)}
	\kern-3pt\overset{\ell > 1/2}{\le} \Vert r_{\square_\ell}(\reviewmath{u^\theta}) \Vert_{H^{-1}(\square_\ell)}
\end{equation*}
\review{and also 
\begin{equation*}
	  \Vert r_{\Omega}(\reviewmath{u^\theta}) \Vert_{H^{-1}(\Omega)} \le \Vert r^\lambda_{\square_\ell}(\reviewmath{u^\theta}) \Vert_{H^{-1}(\square_\ell)}.
\end{equation*}}

For all domains, we have chosen a FEM-triangulation with characteristic length $h=5\cdot 10^{-3}$. We show the results of this experiment in Figure \ref{fig:domainsize}.
\begin{figure}[!htb]
	\centering
	\includegraphics{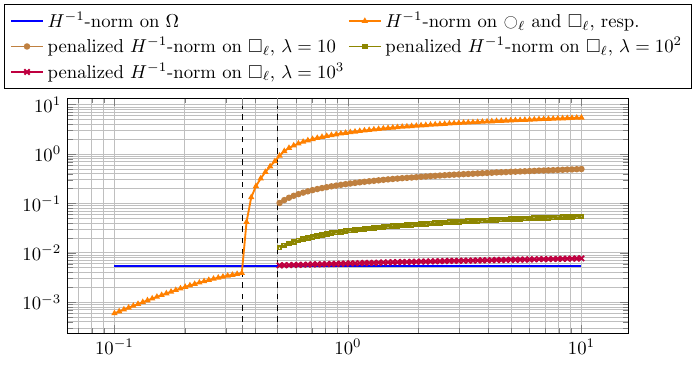}
\caption{\label{fig:domainsize} The $H^{-1}$-norm of the residual on $\Omega$ \review{(horizontal blue line),} \review{and} the \review{corresponding} extended residual \review{norms} on \review{the inner and outer (square) domains of half side length $\ell$.}}
\end{figure}

The plot in Figure\,\ref{fig:domainsize} shows the dependence of the bound on the domain size. The horizontal blue line displays the $H^1_0$-error\footnote{It is equal to the dual norm of the residual on $\Omega$.} of the neural network approximation after training. The vertical dashed lines are the side length $\ell = \sqrt{2}/4$ and $\ell = 0.5$\review{, and indicate the bounds for which $\reviewmath{\mycirc}_\ell \subset \Omega$ and $\Omega \subset \square_\ell$ respectively}, i.e.~the theory from the previous section does not hold in-between those values. The orange curve represents the $H^{-1}$-norm of the extended residual on \review{$\mycirc_\ell$ and }$\square_\ell$. An increasing side length $2\ell$ leads to an increased upper bound for the residual, which is to be expected since the norm is calculated on an increasing domain. However, the dependence is weak for all upper bounds. Also the lower bound deteriorates, when $\ell$ is chosen smaller and smaller. Nevertheless, we can see that the lower bound is sharp for the maximum side length $\ell = \sqrt{2}/4$, for which the lower bound holds. The upper bound without penalization is loose, i.e.~around two orders of magnitudes above the error. Whereas, when we include a penalty parameter $\lambda$, we can see that we can get as close to the true error as we wish. This verifies the theory from above.

In the following, we shall give two applications, for which the bounds would be convenient.

\subsection{Setting for following experiments}
As mentioned in the beginning of this section, possible scenarios of problems solved with NNs include parametric PDEs (PPDEs) and/or PDEs on (spatial) domains with complicated geometries. This also guides our following numerical experiments. Denoting by $\mathcal{P}\subset\mathbb{R}^P$, $P\in\mathbb{N}$, a compact parameter set, we denote by $u^\delta_\mu$ the high-fidelity (but expensive) numerical solution of the PPDE (e.g.\ by finite elements). Then, we trained a NN $\reviewmath{u^\theta}(\cdot;\mu)$ to approximate the solution maps $\mu \mapsto u_\mu$, where $u_\mu$ is the  exact (classical, i.e., pointwise) solution\footnote{Existence and uniqueness of such solutions are assumed under suitable assumptions.} of the PPDE on the domain $\Omega$. For our experiments, we trained the NN using the mean-square loss function
\begin{align*}
	\mathcal{L}(\theta) := \sum\limits_{\mu \in \mathcal{S}_\mathcal{P}} \sum\limits_{x \in \mathcal{S}_\Omega} | u^\delta_\mu(x) - \reviewmath{u^\theta}(x;\mu) |^2,
\end{align*}
where we used finite training data sets $\mathcal{S}_\mathcal{P} \subset \mathcal{P}$ for the parameter and $\mathcal{S}_\Omega \subset \Omega$ for the physical variable (which could also be space and time for parabolic problems). Although our numerical tests use neural networks trained against FE solutions, the proposed certification framework is designed to operate independently of the specific loss and training method — and thus applies equally to, e.g., classical PINNs, variational NNs, and data-driven surrogates. The following experiments should investigate the quantitative behaviour of the error bounds and we have not focused on extensive hyperparameter optimization, i.e. the error of the neural network approximation is not of particular importance here.
\subsection{A parametric domain}\label{parametric domain}
Our next numerical experiment is particularly suited for domain embedding, namely a linear elliptic PDE, posed on a domain $\Omega_\mu$ with parameterized boundary shown in Figure \ref{fig:shapeOptdom}. It is the unit square with a cutout, which cannot be smoothly transformed into a reference domain, due to the sharp corners for angles $\mu > 0$. Such a situation occurs e.g.~in geometry optimization, where PINNs have already been used, \cite{SUN2023116042}. 
\begin{figure}[!htb]
	\begin{minipage}{1.0\textwidth}
		\centering
		
\includegraphics{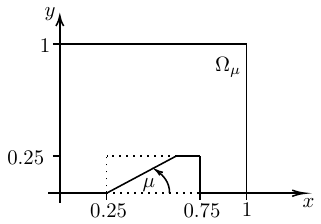}
 	\end{minipage}
	\caption{A parameterized square, where the parameter $\mu$ is the angle of a recess.\label{fig:shapeOptdom}}
\end{figure}
	
When solving a PDE on $\Omega_\mu$ with the finite element method, re-meshing  might be necessary for different $\mu$. On the other hand, the training of a NN, particularly with PINNs, is straightforward by excluding all points outside of the domain. This shows why PINNs might be an attractive method for a PDE on $\Omega_\mu$.

Since the underlying domain is parameterized, we can consider a non-parametric elliptic equation of the form
\begin{equation*}
	\dual{r_{\Omega_\mu}(u_\mu),v}_{H^1_0(\Omega_\mu)} := (A\, \nabla u_\mu, \nabla v )_{L_2(\Omega_\mu)} + (b \cdot \nabla u_\mu + c\, u_\mu, v )_{L_2(\Omega_\mu)} - f_\mu(v) = 0, 
\end{equation*}
for all $v \in H^1_{0}(\Omega_\mu)$ with the diffusion, convection and reaction coefficients given by
\begin{equation*}
	A \equiv \begin{pmatrix}
		1/2 & 1/4 \\
		1/4 & 1/2 
	\end{pmatrix} , \quad 
	b \equiv \begin{pmatrix}
		10 \\
		-3 
	\end{pmatrix}  
	\quad \text{ and } \quad
	c(x,y) := xy+1.
\end{equation*}	
The source function is defined by $f_\mu(v) := (10, v)_{L_2(\Omega_\mu)}$. We choose the training set $\mathcal{S}_{\Omega_\mu}$ as $2^{16}$ random points in $\Omega_\mu$. The parameter training set $\mathcal{S}_\mathcal{P}$ consists of five equidistant points in $\mathcal{P}=[0,\pi/2]$, $P=1$. The extended domain is chosen as $\square:= (0,1)^2$ and the imbedded domain as $\mycirc := (0,1) \times (0.25,1)$.

\review{We first extend $r_{\Omega_\mu}(u^\theta(\mu))\in H^{-1}({\Omega_\mu})$ to
$\bar{r}_{\Omega_\mu}(u^\theta(\mu))\in (H^1({\Omega_\mu}))'$ by setting, for $v\in H^1({\Omega_\mu})$, 
\begin{align*}
\dual{\bar{r}_{\Omega_\mu}(u^\theta(\mu)),v}_{H^1({\Omega_\mu})}:= 
&(A\nabla u^\theta(\mu),\nabla v)_{L_2({\Omega_\mu})}+\\
&+(b\cdot\nabla u^\theta(\mu)+cu^\theta(\mu),v)_{L_2({\Omega_\mu})}
-\dual{f_\mu,v}_{H^1({\Omega_\mu})},
\end{align*}
for which it holds that
$
\bar{r}_{\Omega_\mu}(u^\theta(\mu))|_{H^1_0({\Omega_\mu})}=r_{\Omega_\mu}(u^\theta(\mu)).
$ Then, t}he extension and restriction of the residual can be defined as 
\begin{align*}
	\dual{r_\square(\reviewmath{u^\theta}(\mu)), v }_{H^1_0(\square)} 
    &:= \dual{\reviewmath{\bar{r}}_{\Omega\reviewmath{_\mu}}(\reviewmath{u^\theta}(\mu)), v\reviewmath{|_{\Omega_\mu}}}_{H^1(\reviewmath{\Omega_\mu)}}, 
    &&\forall v \in H^1_{0}(\square), \\
	\dual{ r_\mycirc(\reviewmath{u^\theta}(\mu)), v }_{H^1_0(\mycirc)} 
    &:= \dual{r_{\Omega\reviewmath{_\mu}}(\reviewmath{u^\theta}(\mu)),E_{\reviewmath{\mycirc \rightarrow}{\Omega\reviewmath{_\mu}}} v}_{H^1_0({\Omega\reviewmath{_\mu}})},  
    &&\forall v \in H^1_{0}(\mycirc).
\end{align*}
Again, the penalized residual extension $r^{\lambda}_\square(\reviewmath{u^\theta}(\mu))$ is defined as in \eqref{eq:computablePenResidual}\review{, and $r^{\lambda}_\square(\reviewmath{u^\theta}(\mu))\in \mathcal{E}(R'_{{\Omega_\mu}\rightarrow\square} r_\Omega;H^{-1}(\square))$. Consequently,} Corollary\,\ref{cor:boundsSobolev} applies and yields the error estimate
\begin{equation*}
	\frac{\Vert r_\mycirc(\reviewmath{u^\theta}(\mu)) \Vert_{H^{-1}(\mycirc)}}{\Vert A \Vert_{L_\infty} + \Vert b \Vert_{L_\infty} + \Vert c \Vert_{L_\infty}} 
    \le \Vert \nabla u_\mu - \nabla \reviewmath{u^\theta}(\mu) \Vert_{L_2(\Omega\reviewmath{_\mu})} 
    \le \frac{\Vert r_\square(\reviewmath{u^\theta}(\mu)) \Vert_{H^{-1}(\square)}}{\lambda_{\min}(A)},
\end{equation*}
\review{and also
\begin{equation*}
\Vert \nabla u_\mu - \nabla \reviewmath{u^\theta}(\mu) \Vert_{L_2(\Omega\reviewmath{_\mu})} 
    \le \frac{\Vert r^\lambda_\square(\reviewmath{u^\theta}(\mu)) \Vert_{H^{-1}(\square)}}{\lambda_{\min}(A)}
\end{equation*}
}
from which we deduce the constants.

\begin{figure}[!htb]
	\centering
\includegraphics{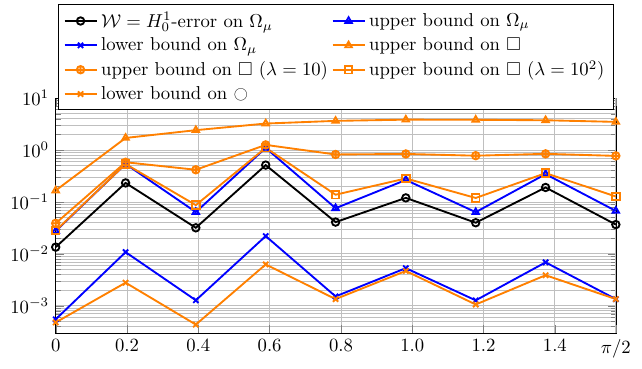}
\caption{\label{fig:param_dependencev2} Absolute $H^1_0$-error as well as the estimated $H^1_0$-error on $\Omega$, $\square$ and $\mycirc$. The horizontal axis corresponds to the angle $\mu$.}
\end{figure}
The results for a set of nine parameters, which serve as the test set for the NN solution, are depicted in Figure \ref{fig:param_dependencev2}. The Riesz bounds on $\Omega_\mu$ (blue) follow the slope of the exact error (black), the upper bound being quite sharp, the lower one too optimistic by a factor of about 10. The lower bound for the error estimator (orange), calculated on $\mycirc$, follows the slope and is quite sharp. The reason might be that the cut-off region affects neither the Riesz lower bound nor the one on $\mycirc$. On the other hand, however, the upper bound does not follow the slope of the error and is too pessimistic by a factor of 10. The change in the geometry has only a small effect on the proposed upper error bound. The situation changes when we include the penalized upper bound. We expect the penalized upper bound to be sharper for increasing $\lambda$. This is indeed the case and, as in the previous example, one could choose $\lambda=10^2$ to get a sharp upper bound and still carry out the calculation on the square domain.

\reviewNR{Since the penalized Riesz problem contains the discontinuous coefficient $1+\lambda\chi_{\square\setminus\Omega_\mu}$, we additionally investigate the sensitivity of the \sloppy computed estimator w.r.t.~numerical quadrature. We fix $\lambda=10^2$ and a standard $P_1$ finite-element space on a uniform $512\times512$ triangulation of $\square=(0,1)^2$, which is not fitted to $\partial\Omega_\mu$, comprising approximately $2.6\times 10^5$ degrees of freedom. The background mesh is kept fixed as $\mu$ varies. For each $\mu\in\mathcal S_{\mathcal P}$, the discrete residual load vector is assembled once using a Gauss quadrature rule of order $q_{\mathrm{load}}=8$.
Subsequently, we vary only the uniform Gauss quadrature order $q$ used to integrate the characteristic function in the penalty term. On elements intersected by $\partial\Omega_\mu$, membership in $\Omega_\mu$ is evaluated pointwise at the quadrature points, while no remeshing or local refinement is employed. 

For fixed $\lambda$, we denote by $\eta_{h,q}(\mu)$ the penalized upper estimator computed on a fixed background discretization of characteristic length $h$ and with penalty quadrature order $q$. We measure its quadrature sensitivity by the relative error 
\[
E_{p,\mathrm{rel}}(q)
:=
\frac{\|\eta_{h,q}-\eta_{h,q_{\mathrm{ref}}}\|_{\ell^p(\mathcal{S}_\mathcal{P})}}
     {\|\eta_{h,q_{\mathrm{ref}}}\|_{\ell^p(\mathcal{S}_\mathcal{P})}},
\qquad p\in\{2,\infty\}.
\]
Taking $q_{\mathrm{ref}}=64$ as a high-order quadrature reference, Figure~\ref{fig:q_stab} shows that the estimator stabilizes rapidly: already for $q=2$ the relative discrepancies are of order $10^{-4}$, and they decrease further for $q\geq 4$. Therefore, for this parametric geometry, the discontinuous characteristic coefficient can be accurately integrated on a fixed uniform background mesh with low-order quadrature, without an interface-fitted mesh or local refinement near $\partial\Omega_\mu$.}
\begin{figure}[!htb]
    \centering
    \includegraphics[width=0.82\textwidth]{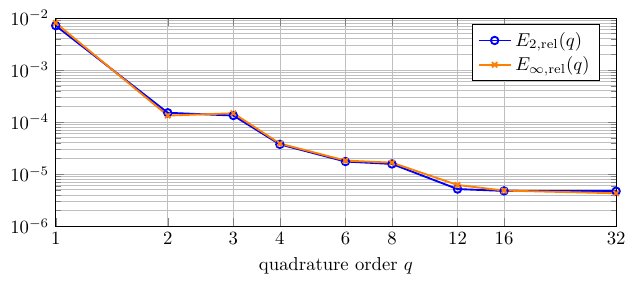}
    \caption{Quadrature sensitivity of the penalized upper estimator for the parameter-dependent domain: relative $\ell^2(\mathcal{S}_\mathcal{P})$- and $\ell^\infty(\mathcal{S}_\mathcal{P})$-discrepancies with respect to the high-order reference $q_{\rm ref}=64$. The background FE discretization of $\square$ and the penalty parameter $\lambda=10^2$ are kept fixed, while only the quadrature order $q$ used for the characteristic coefficient is varied.}
    \label{fig:q_stab}
\end{figure}

\reviewNR{We next isolate the spatial-discretization effect by fixing $\lambda=10^2$ and $q=64$ and uniformly refining the same unfitted background discretization of the $\square$ domain. For $N=128,256,512,1024$ subdivisions in each coordinate direction, corresponding to $h=1/N$, we compare the penalized estimator with reference Riesz estimates on $P_1$-triangulations of $\Omega_\mu$ with $h=1/1024$, denoted simply by $\eta_{\text{ref}}(\mu)$. Defining
\[
\delta_{h}(\mu)
:=
\frac{\eta_{h,64}(\mu)-\eta_{\text{ref}}(\mu)}
     {\eta_{\text{ref}}(\mu)},
\]
we obtain
\[
\begin{array}{c|cccc}
N & 128 & 256 & 512 & 1024 \\ \hline
\displaystyle\min_{\mu\in\mathcal{S}_{\mathcal{P}}}\delta_h(\mu)
& -2.01\times10^{-1}
& -6.59\times10^{-2}
& -8.12\times10^{-3}
& 4.09\times10^{-14}
\\[4pt]
\displaystyle\max_{\mu\in \mathcal{S}_{\mathcal{P}}}\delta_h(\mu)
& 2.4657\times10^{-1}
& 3.6011\times10^{-1}
& 4.4711\times10^{-1}
& 4.9325\times10^{-1}
\end{array}
\]
The worst relative underestimation with respect to the reference Riesz estimates on $\Omega_\mu$ decreases systematically under uniform refinement and disappears for $N=1024$. The quantity $\max_{\mu\in\mathcal S_{\mathcal P}}\delta_h(\mu)$ reflects the
largest positive gap between the discrete penalized estimator and the corresponding reference Riesz estimate, for fixed $\lambda$. Together with the preceding quadrature study, these results indicate that the quadrature error associated with the discontinuous characteristic coefficient can already be made negligible using low-order quadrature, while sufficient spatial resolution can be achieved by standard uniform refinement of the simple enveloping domain $\square$.

We emphasize that this Riesz problem is an auxiliary problem for the
numerical evaluation of the a posteriori estimator. Our aim is not to develop an optimal discretization method for the resulting interface problem. For the present parametric geometry, the experiments show that the desired accuracy can be reached by a straightforward uniform mesh on the simple $\square$ domain, together with low-order quadrature, without interface-fitted remeshing or local refinement near $\partial\Omega_\mu$. Since the auxiliary problem is posed on the geometrically simple domain $\square$, such uniform discretizations retain a structured mesh hierarchy and are amenable to standard efficient iterative and multilevel solvers without requiring geometry-specific discretizations of $\Omega_\mu$. A systematic investigation of the interplay among the background mesh size $h$, the quadrature order $q$, and the penalty parameter $\lambda$ is beyond the scope of the present work.}

\begin{figure}[!htb]
  \centering
  \includegraphics[width=\textwidth]{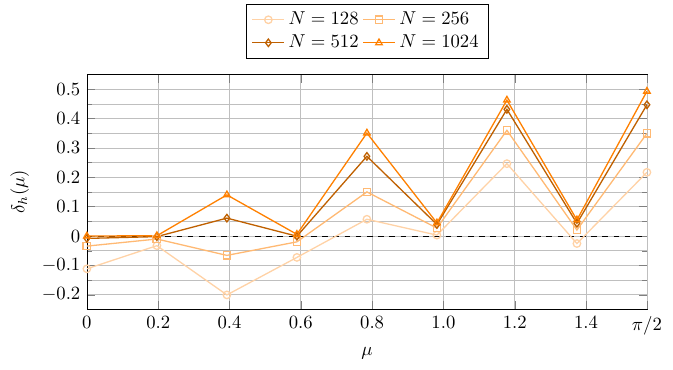}
  \caption{Spatial-discretization study of the penalized upper estimator for $\lambda=10^2$. The signed relative discrepancy $\delta_h(\mu)$ is shown over the parameter set for uniform meshes with $h=1/N$, where $N$ denotes the number of subdivisions in each coordinate direction. The dashed line indicates zero discrepancy with respect to the reference Riesz estimate on $\Omega_\mu$.}
  \label{fig:parametric-spatial-resolution}
\end{figure}

\subsection{A parabolic problem on a non-convex polytope}
Finally, we report results for a time-dependent problem on a domain with \enquote{complicated} geometry. To this end, consider the parameterized parabolic problem $\dot{u}_\mu+A_\mu u_\mu = f$, $u_\mu(0)=0$ on $Q:=I\times\Omega$, $I:=(0,1)$ being the time horizon and $\Omega\subset\mathbb{R}^2$ is the map of the state of Arkansas\,(USA) depicted in Figure \ref{fig:domain_arkansas}. The domain is a non-convex polytope and therefore has a Lipschitz-boundary.  Moreover, $\Omega$ has sharp corners on the right- and lower left-hand side. 
\begin{figure}[!htb]
	\begin{minipage}{1\textwidth}
		\centering
		\includegraphics[width=0.5\textwidth]{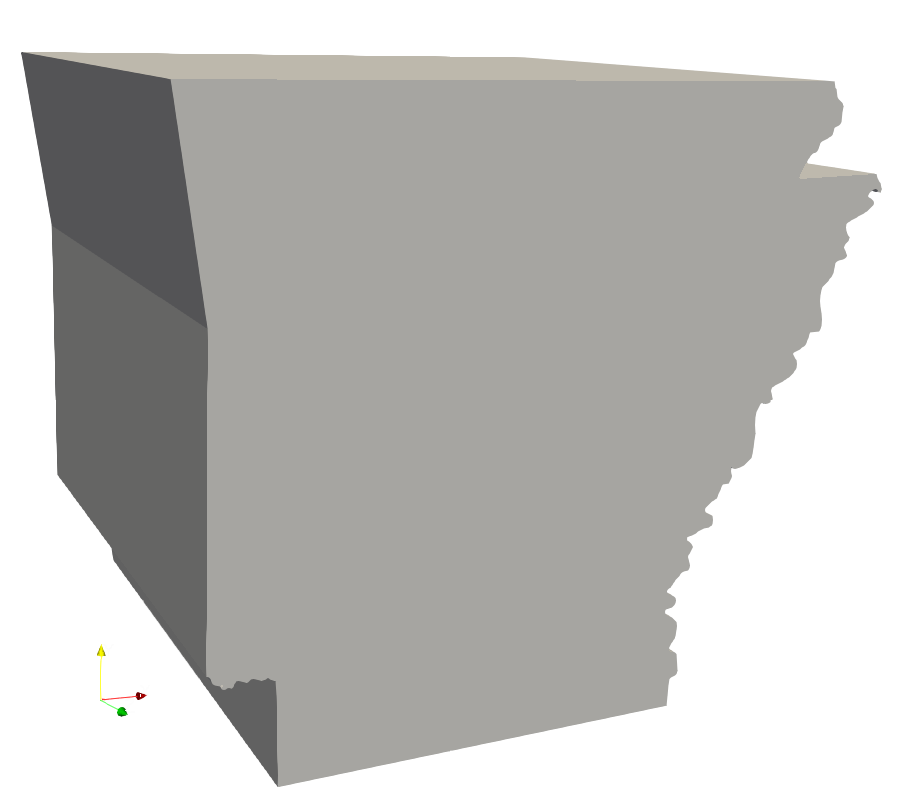}
	\end{minipage}
	\caption{The space-time domain $Q = I\times \Omega$, where the green axis refers to the time.\label{fig:domain_arkansas}}
\end{figure}
The parametric elliptic operators $A_\mu \in \mathcal{L}(H^1_0(\Omega), H^{-1}(\Omega))$ are defined by the variational form 
\begin{equation*}
	\langle A_\mu \varphi, \psi \rangle_{H^1_0(\Omega)} 
    := (K \nabla \varphi, \nabla \psi )_{L_2(\Omega)} + (b_\mu \nabla \varphi + c \varphi, \psi )_{L_2(\Omega)}, \quad \forall \, \varphi, \psi \in H^1_0(\Omega),
\end{equation*}
where the time-independent coefficient functions are chosen as
\begin{equation*}
	K \equiv \begin{pmatrix}
		1 & 0 \\
		0 & 0.1 
	\end{pmatrix} , \quad 
	b_\mu(x,y) := 5\begin{pmatrix}
		 \sin(\mu y)\\
		\cos((x+1)^{\mu/8}) 
	\end{pmatrix} \text{ and } \quad
	c(x,y) := xy+1,
\end{equation*}	
with the parameter set $\mathcal{P}:= [1,10] \subset \mathbb{R}$, $P=1$. The parameter dependent convection is chosen to be \emph{not} affinely decomposable, so that this standard assumption of the reduced basis method is not valid, see e.g.\,\cite{RozzaRB}, and training a NN seems attractive.  We use the space-time variational formulation of the parabolic PDE as introduced in Example \ref{Ex:SpaceTime} above.

The height of Arkansas is normalized to $1$ and the front upper left corner is located at $(0,1,0)^T \in \mathbb{R}^3$, so that we define $\square := (0,1.2) \times (0,1)$ and $\mycirc := (0.1345,0.783) \times  (0,1) \subset \Omega$. With this setting, the (penalized) extension operator from Proposition\,\ref{prop:LisBochner} can be used and Corollary\,\ref{cor:boundsBochner} applies.

The NN has again been trained with a high-fidelity finite-element solution using training sets $\mathcal{S}_Q$, consisting of $1.5 \cdot 10^5$ points for the space-time domain $Q$ and $\mathcal{S}_\mathcal{P}$, consisting of $16$ equidistant points for the parameter. The error has been measured on $31$ parameters. 

\begin{figure}[!htb]
	\centering
\includegraphics{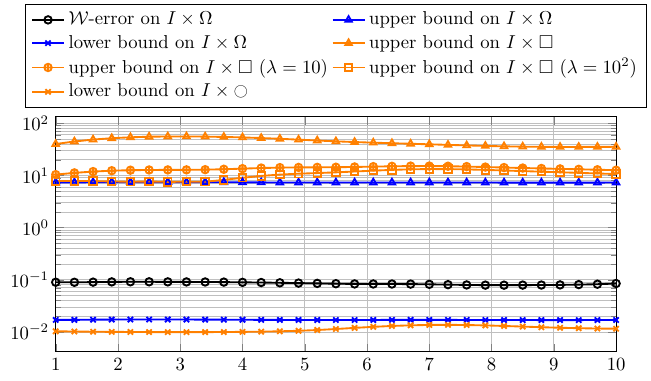}	
\caption{\label{fig:param_dependencev3} Absolute $\mathcal{W}$-error as well as the estimated $\mathcal{W}$-error on $I\times \Omega$, $I\times\square$ and $I\times\mycirc$. The horizontal axis corresponds to the parameter value $\mu \in \mathcal{P}$.}
\end{figure}
The results are depicted in Figure \ref{fig:param_dependencev3}. The error of the NN approximation depends only slightly on the parameter $\mu$. This is also reflected by the bounds, so that they are basically multiples of the true error. The Riesz bounds on $Q=I\times\Omega$ are about a factor 70 off the true error (black). This is due to the constants in the error-residual relation, see \cite[Thm.\ 5.1]{schwab2009space} and the appendix therein. The (penalized) error bounds (orange) are quite sharp for $\lambda=10^2$. Also the lower bound has nearly no offset. Again, the data is smooth on $I\times\mycirc$, so that we can use an efficient spectral method as we do not need to resolve the complicated geometry of $\Omega$.

\subsection{Computational times}
In order to investigate the computational overhead required for the lower and upper bounds, we collect in Table \ref{Tab:CPU} the CPU/GPU times for (i) the NN training, (ii) the evaluation of the NN at the points $\lbrace \mu \rbrace \times \mathcal{S}_\Omega$ for one $\mu \in \mathcal{S}_\mathcal{P}$ and (iii) solving the Riesz representation problems on $\mycirc$ and $\square$ also for one $\mu\in \mathcal{S}_\mathcal{P}$. The times with respect to all $\mu \in \mathcal{S}_\mathcal{P}$ scale linearly. The time measurement has been carried out using standard devices for each task, e.g., a NVIDIA Tesla V100 GPU has been used for the training and the evaluation of the NNs. The termination criterion of the training process has been a maximum number of iterations, which was $5 \cdot 10^3$ in the case of domain \reviewNR{size dependence} and $10^4$ iterations for the parametric and Arkansas domain. The evaluation of the error estimator on $\square$ has been parallelized using two Intel Xeon Gold 6230 CPUs with 20 cores each. For the discretization, we have used standard $P1$-finite elements with around $2.5 \cdot 10^5$ degrees of freedom. The Riesz problems on $\mycirc$ have been solved with the spectral method using a nodal Lagrange basis of order 12 in the 2D case and of order 8 in the time-dependent case. This leads to small dense linear systems, which can be solved using serial direct solvers and the error can be expected to be near machine accuracy. In case of the Arkansas domain the matrix size was $729 \times 729$ and $169 \times 169$ in the stationary cases. Thus, the comparable long solving time of $2.249$ seconds may be due to the non-optimal internal routines of FEniCSx. 

We can see that the time for the error estimation is negligible in comparison to the training time. Moreover, the lower bound on $\mycirc$ can be computed efficiently and this can even be improved if one would use specialized spectral solvers. The numbers confirm the efficiency of the method, even though we did not even use a highly optimized implementation, but the standard algorithms within the FEniCSx implementation.

\begin{table}[htbp]
    \centering
    \begin{tabular}{l|r|r|r|r}
        \multicolumn{3}{c}{} & 
        \multicolumn{2}{|r}{{\textbf{Error estimation}}}\\
         \textbf{Problem} 
         & \textbf{Training} 
         & \textbf{Evaluation}
         & $\mycirc$ & $\square$ \\ \hline
         Parametric domain & 926.76 & 0.0128 & 0.011 & 0.15\\ \hline
         Arkansas & 7536.71 & 0.0720 & 2.249 & 1.77\\ \hline
    \end{tabular}
    \caption{Times (in seconds) for training (GPU), NN evaluation (GPU) and error estimation (CPU) on $\mycirc$ and $\square$.}
    \label{Tab:CPU}
\end{table}

 \section{Summary and outlook}
In this work, we introduced a novel framework for rigorously certifying the accuracy of \revKU{any given approximation to the solution of a PDE} through a posteriori error estimation.~\revKU{The approximation is seen as a black box determined e.g.\ by a commercial code or a} neural network (NN). The proposed method constructs efficient, computable error bounds in the natural norm imposed by well-posedness of the PDE formulation, under minimal regularity assumptions, is applicable to complex geometries, and is independent of the training process and the choice of loss function.

We derived both upper and lower error bounds that were shown to be reasonably sharp while enhancing computational efficiency. This is achieved by replacing the potentially complex PDE domain by a simpler reference domain and by establishing relationships between functionals defined w.r.t.~the corresponding spaces on the embedded domains. These relationships yield the desirable bounds, accessed by the evaluation of dual norm estimates using Riesz representation as an alternative to using wavelet methods as introduced in \cite{ernst2024certified}. 

A central difficulty was the construction of norm-preserving extensions of residuals. Although the existence of such extensions is guaranteed by the Hahn–Banach theorem, they are not computable in general. In a Hilbert space setting, this difficulty can be addressed by exploiting the Riesz representation. We introduced penalized extensions that can be realized through the solution of elliptic problems on the simple enveloping domain. We proved that these extensions converge to the optimal norm-preserving extension as the penalty parameter tends to infinity.
 
Our numerical experiments validate the proposed certification framework for parameter-dependent linear elliptic and time-dependent parabolic PDEs. The numerical investigation of the box size and the sharpness of the upper bound shows a weak dependence. In the case of an elliptic PDE with a parameterized boundary, our framework efficiently certifies \revKU{black box} approximations without requiring a re-meshing for each parameter. Similarly, for a linear parabolic equation in a simultaneous space-time variational formulation, our method successfully extends to Bochner spaces. 

Certifying solutions that only weakly or approximately satisfy boundary data would require a generalization of the underlying functional setting (extension and restriction construction) beyond $H_0^1(\Omega)$. This is an interesting direction for future research and would allow extending the certification framework to more general PINN architectures or learning settings where strict enforcement of boundary conditions is not guaranteed.

Our numerical experiments employ penalized extensions of residual functionals, which provide a computable surrogate of optimal norm-preserving extensions. Moreover, the extension operator used here is optimal for $\mathcal{Y}\revKU{(\Omega)} = L_2\revKU{(\Omega)}$, motivating the use of a first-order system formulation. This ensures that the residual remains in $L_2(\Omega)$, improving even further the accuracy of the certification. Ongoing work includes adaptive wavelet-based approximations in $W_0^{m,p}(\square)$, which allow for error certification in Banach spaces beyond Hilbert settings. A key direction for our future research is to extend the methodology to nonlinear PDEs, where error certification is more challenging but crucial for practical applications.

\begin{appendix}

\section{NNs for solving PDEs} \label{sec:PINNs}
\revKU{Our above framework was motivated by the aim to certify} the solution of a partial differential equation (PDE) using neural networks. Thus we briefly introduce the main concept of solving PDEs using neural networks without going into details, referring to classical PINNs only as an example to fix ideas.

\subsection{PDEs in classical form}
Let $\Omega \subset \mathbb{R}^d$ be a bounded open domain and let $m\in\mathbb{N}$ denote the order of the PDE ($m=2$ for Laplace's equation). We denote by
\begin{align*}
    B^\circ: C^m(\Omega)\to C^0(\Omega)
\end{align*}
the classical (point-wise) form of the differential operator under consideration. Then, given some $f^\circ\in C(\Omega)$, we call $u\in C^m(\Omega)$ a \emph{classical solution} if 
\begin{equation} \label{eq:clPDE}
	(B^\circ u)(x) = f^\circ(x), \quad \forall x \in \Omega,
\end{equation}
where we assume that proper boundary and/or initial conditions are incorporated into the definition of the operator. 
Next, we are given some approximation $u^\delta$ to $u$, e.g.\ in terms of a neural network, and define the (classical) \emph{residual} of \eqref{eq:clPDE} by
\begin{align}\label{eq:residual}
    r_\Omega^\circ (u^\delta)(x) := \revKU{f^\circ}(x) - (B^\circ u^\delta)(x), \quad x\in \Omega. 
\end{align}
Given $u^\delta$, the residual is in principle computable by inserting $u^\delta$ into the PDE operator. However, such classical solutions often do not exist, depending on the data $B^\circ$, $f^\circ$ and $\Omega$. It is well-known that well-posedness of such problems is usually linked to suitable variational formulations as described in Section \ref{sec:VarFormPDEs}.

\subsection{Neural Network Approximation}
Neural networks seem suitable for solving PDEs because they are universal function approximators, see \cite{Hornik1991,Hornik1989,Guehring2019}. In the following we give an example of a neural network. 
\subsubsection*{Neural networks (NNs)} 
The notation of NNs in this paragraph is based upon \cite{Berner2021,Gribonval2021,Petersen2018}. A NN is a function $\Phi_a(\cdot;\theta):\mathbb{R}^{N_{0}} \rightarrow \mathbb{R}^{N_{L}}$, where $a$ is the \emph{architecture} and $\theta$ are the \emph{parameters}. Both the architecture and parameters determine the input-output function $\Phi_a(\cdot;\theta)$ of the NN. In case of a feed-forward NN, the architecture $a=(N,\rho)$ can be described by the vector of neurons per layer $N=(N_0,...,N_L)\in\mathbb{N}^{L+1}$, where $L \in \mathbb{N}$ denotes the number of \emph{layers} ($N_0$ being the input and $N_L$ the output dimension) and the \emph{activation function} $\rho: \mathbb{R} \rightarrow \mathbb{R}$.

The parameters of the NN read $\theta = (W^{(l)},b^{(l)})_{l=1,...,L}$, where $W^{(l)} \in \mathbb{R}^{N_{l} \times N_{l-1}}$ are the \emph{weight matrices} and $b^{(l)} \in \mathbb{R}^{N_{l}}$ are called \emph{bias vectors}. 
The output $\Phi_a(z;\theta)$ of the NN for an input $z\in\mathbb{R}^{N_0}$ is then defined as $\Phi_{a}(z;\theta) := \Phi^{(L)}(z;\theta)$, where
\begin{align*}
	\Phi^{(1)}(z;\theta) &= W^{(1)} z + b^{(1)}, \\
	\hat{\Phi}^{(l)}(z;\theta) &= \rho(\Phi^{(l)}(z;\theta)), \quad l=1,...,L-1, \quad \text{and} \\
	\Phi^{(l+1)}(z;\theta) &= W^{(l+1)} \hat{\Phi}^{(l)}(z;\theta) + b^{(l+1)}, \quad l = 1,...,L-1,
\end{align*}
and $\rho$ is applied component-wise. In the following the architecture is omitted in the notation as we view it as being fixed once and for all.\smallskip\\ 

For a neural network $\reviewmath{u^\theta} := \Phi(\cdot;\theta)$ to approximate the solution of a PDE, the parameters $\theta$ must be determined, which is done in the training phase. Thereby, a tailored minimization problem is defined, with which the parameters are \emph{learned}.

\subsubsection*{Example: Classical PINNs.}
The core idea of PINNs is to use the residual for the definition of a loss function within the training of a neural network. Let us briefly describe this for the above classical solution concept, even though (i) there are several other approaches in the literature and (ii) our error analysis does not depend on the choice of the loss function.

In the regime of PINNs the function to be minimized involves the PDE (e.g.\ in classical form \eqref{eq:clPDE}), which is the reason for the name \emph{physics-informed}. With respect to the classical form of the PDE, a set of sample points $\mathcal{S}_\Omega$ in $\Omega$ are chosen and \eqref{eq:clPDE} is posed only for those points. This leads to the definition of the loss function
\begin{equation}\label{eq:lossfunc}
	\mathcal{L}(\theta) 
    := {\frac{| \Omega |}{| \mathcal{S}_\Omega|}} \sum_{x \in \mathcal{S}_\Omega} 
    \vert (r^\circ_\Omega (\reviewmath{u^\theta}))(x) \vert^2 {\approx \int\limits_{\Omega} \vert (r^\circ_\Omega (\reviewmath{u^\theta}))(x) \vert^2 \, dx},
\end{equation}
\revKU{where the latter  approximation is valid as long as the residual is square integrable.} 
Often, an additional sampling is required for satisfying the boundary conditions. We shall assume that $\reviewmath{u^\theta}$ satisfies given boundary conditions as this can be achieved with the aid of approximate distance functions, see \cite{sukumar2022exact}. Other approaches tackling this issue include \cite{makri,makri1}.

Such classical PINNs have e.g.\ been investigated for a broad scope of linear and nonlinear PDEs, see e.g. \cite{Berg2018,Raissi2019,Mao2020,Cai2021,Cai2021a,Hu2024}. The main advantages of the method are its straightforward applicability and, due to the sampling of $\mathcal{S}_\Omega$, it results in a mesh-free approximation.

Variants such as VPINNs, Deep Ritz, and Deep Galerkin methods use alternative loss constructions, often based on variational formulations.

\subsubsection*{Objective: Certification of NN approximations}
Now, let $u^\theta$ denote an approximate solution parameterized by a neural network with weights $\theta \in \mathbb{R}^N$. Our framework places \emph{no restriction} on how $u^\theta$ is obtained. It may arise from classical PINNs, VPINNs, Deep Ritz methods, Deep Galerkin methods, pre-trained surrogate models, or any other black-box neural network approximator. We refer to all such approximations under the umbrella term \emph{neural network approximations}.

Given a NN approximation $\review{u^\theta}$, our goal is to certify the approximation quality of $\review{u^\theta}$ a posteriori, i.e., to bound the error $\|u - \review{u^\theta}\|$ from above and below, independently of the  loss functional or training process.

\end{appendix}

\subsubsection*{Acknowledgement} The authors acknowledge support by the state of Baden-W\"urt\-tem\-berg through bwHPC, and the Hellenic Foundation for Research and Innovation (H.F.R.I.) under the \enquote{2nd Call for H.F.R.I.~Research Projects to support Post-Doctoral Researchers} (project number: $01247$). 

\bibliographystyle{ieeetr}


\end{document}